\documentclass[a4paper, 11pt, reqno]{amsart}
\usepackage{mystyle_arXiv}
\usepackage[natbibapa]{apacite}

\newcommand{\cU}{{\mathcal U}}
\renewcommand{\veth}{\phi^\varepsilon}

\title{A duality and free boundary approach to adverse selection$^*$}

\thanks{$^*$ Robert McCann's research is supported in part by the Canada Research Chairs program CRC-2020-00289, the Simons Foundation, and Natural Sciences and Engineering Research Council of Canada Discovery Grants RGPIN--2015--04383 and 2020--04162. The work of Kelvin Shuangjian Zhang is supported by the ERC project NORIA. The authors are grateful to Cale Rankin$^\dagger$ for providing Appendix \ref{section:Rankin}, to Jean-Marie Mirebeau for permission to reproduce one of his figures, to Luis Caffarelli for making his unpublished notes with Lions available to us, and to Ivar Ekeland, Jean-Guillaume Forand, Xianwen Shi, Alex Kolesnikov, and participants on our virtual presentation of these results at the Moscow Seminar on Mathematical Problems in Economics in May 2021 for useful feedback,  and to Toronto's Fields Institute for the Mathematical Sciences, where much of this work was performed. 
	\copyright \today}

\author[1]{Robert J. McCann$^\dagger$}\thanks{$^\dagger$Department of Mathematics, University of Toronto, Toronto, Ontario, Canada, M5S 2E4 {\tt mccann@math.toronto.edu} and {\tt cale.rankin@utoronto.ca}}
\author[2]{Kelvin Shuangjian Zhang$^\ddagger$}\thanks{$^\ddagger$
	School of Mathematical Sciences,
	Fudan University, Shanghai,
	CHINA 200433
 {\tt ksjzhang@fudan.edu.cn}}

\allowdisplaybreaks

\begin{document}

\begin{abstract}
Adverse selection is a version of the principal-agent problem that includes monopolist nonlinear pricing,
where a monopolist with known costs seeks a profit-maximizing price menu facing a population of potential consumers
whose preferences are known only in the aggregate.  For multidimensional spaces of agents and products,  
\cite{RochetChone98} reformulated this problem as a concave maximization over the set of convex functions,
by assuming agent preferences combine bilinearity in the product and agent parameters with a quasilinear sensitivity to prices.
We characterize solutions to this problem by identifying a dual minimization problem.  This duality allows us to reduce
the solution of the square example of Rochet-Chon\'e to a novel free boundary problem,  giving the first analytical description
of an overlooked market segment.
\end{abstract}

\maketitle
\vspace{-0.5cm}
\vspace{0.5cm}

\tableofcontents

\section{Introduction}\label{section:introduction}

The principal-agent problem has provided an important framework for modelling economic questions involving asymmetric information since the 1970s.
In the context of nonlinear pricing, the principal represents a monopolist who wishes to maximize her total profit over all possible price menus, facing a given distribution of agent (i.e., consumer) types, while each consumer aims to optimize his utility by choosing one product anonymously and paying its price to the monopolist.

The monopolist faces a bi-level optimization problem. Every time she changes the price menu, the consumers' choices of products may change in response, resulting in a different distribution of the products sold and a corresponding change to the monopolist's profit. However, this bi-level optimization can be reformulated as a {(single-level)} problem, with nonlinear constraints on the product-price pair to enforce incentive compatibility and individual rationality. The former condition ensures that the product-price pair reflects the choices of consumers facing the price menu. The latter reflects the existence of an outside option whose price the monopolist cannot control. For example, public transportation might represent an outside option relative to a vehicle-selling monopolist.

Under suitable assumptions on the consumers' direct utility, this problem can also be reformulated as a maximization problem with generalized convexity constraints on indirect utilities. This reformulation exploits the natural duality between the monopolist's price menu and the agents' indirect utilities, and the implementation result that each consumer's best choice lives in a generalized subdifferential of their indirect utility function. For unidimensional consumer types, this dual approach can be traced back to \cite{Mirrlees71} work on optimal taxation.  
A dual approach for multidimensional consumer types with bilinear preference functions was developed by \cite{RochetChone98},
in a landmark contribution among the vast subsequent literature on mechanism design with multidimensional
types.  
Analogous implementability, existence, and stability of optimal strategies for more general quasilinear preferences can be found in \cite{Rochet87}, \cite{Carlier01} and \cite{FigalliKimMcCann11} respectively.
Such results were recently extended to fully nonlinear preferences by \cite{NoldekeSamuelson18} and
\cite{McCannZhang19}. A control-theoretic approach to the quasilinear case was developed by \cite{Basov05}.

The early literature focuses on the one-dimensional version of such questions,  where products are parameterized by quality and agents by wealth, as in the classical studies of \cite{Mirrlees71} on taxation and \cite{Spence74} on educational signalling. 
In some cases, explicit solutions can be obtained on interval domains, as in \cite{MussaRosen78}. Here, the principal's optimization separates the domain of the agents into two parts: a bottom region where the participation constraint binds; and a top region where agents choose customized products according to their types. All the agents choose the same product in the bottom part: the outside option. In one-dimensional cases, the fraction of types choosing the outside option may be positive or zero, whereas 
for multidimensional strictly convex sets of types, \cite{Armstrong96} shows the non-participation region must have a positive measure. 

Multidimensional versions of the problem,  in which both agents and products require several variables to describe,  have proven much thornier to analyze; see e.g.~\cite{McAfeeMcMillan88} or \cite{Wilson93}. 
Explicit solutions are extremely difficult to obtain except on radially symmetrical domains (\cite{Zhang18}).   One must now solve {\em partial} 
in addition to {\em ordinary} differential equations,
subject to the nonstandard convexity constraint arising from incentive compatibility.
For example, in the case of bilinear preferences on the plane, the {\em monotone} (scalar-increasing) map from agents to products representing agents' optimal choices must be replaced by the {\em gradient} of the agent's {\em convex}
indirect utility. 
Furthermore, the optimal solution has regions displaying different behaviour according to the rank of the Hessian of this convex function, as discovered by  \cite{RochetChone98}.
Between \cite{Armstrong96}'s positive bottom fraction of agents who select the outside option,
and the product-customizing top market segment, (where the Hessian matrix of indirect utility has {zero versus full rank} respectively), 
a bunching region can lie,  which is
foliated by families of agents (isochoice sets) who select the same product type in the optimal solution,
as Rochet and Chon\'e discovered in their two-dimensional square model. The indirect utility in this region has a Hessian matrix of rank (and corank) $1$. Moreover, the Euler-Lagrange equation
of the optimization takes on a different character in each of these regions, so that an analytical solution to the problem requires matching  (or ``smoothly pasting'') a solution of a partial differential
(Poisson) equation in the top region, to the solution of an ordinary differential equation in the bunching region.  Finding the boundary between these regions,
whose geometry is a priori unspecified,  becomes part of the problem: we shall show it does not generally reduce to a point (as in one-dimension),  nor to a line {(or hyperplane)} as 
Rochet and Chon\'e {assumed.} 
Finally, the problem requires appropriate boundary conditions and is highly sensitive to the shape of the domain.

The contributions of the present work are two-fold.  First, we develop a duality theory which characterizes the solution to the multidimensional adverse selection problem,
under Rochet and Chon\'e's 
assumptions of bilinearity of agent preference in product type, and quasilinearity in price.  Second,  we introduce a new free-boundary problem which characterizes the solution
to the Rochet-Chon\'e square example analytically.   This requires us to derive an Euler-Lagrange equation for a segment of the market overlooked by Rochet and Chon\'e,
in which the isochoice segments vary in slope as well as in length,  as suggested by numerical simulations of \cite{Mirebeau16}.

Duality has proved to be a powerful tool for characterizing solutions to other revenue optimization problems. For instance, \cite{DaskalakisDeckelbaumTzamos17} developed a strong duality theory to find the optimal mechanism for selling multiple goods to a single additive buyer, generalizing the single good auction of \cite{Myerson81}.  Later \cite{KleinerManelli19} provided another 
approach to this duality. \cite{GiannakopoulosKoutsoupias2018} studied the optimal (auction) strategy for selling multiple goods to multiple buyers and found a (different) duality theory for the single bidder case. 
A duality approach for multi-bidder multi-item auctions was discovered by
\cite{KolesnikovSandomirskiyTsyvinskiZimin22+} in parallel with the present manuscript;
they interpret their dual as a continuous optimal flow problem whose prescribed divergence second-order stochastically
dominates a certain neutral measure inferred from the data.
We hope to convince the reader that the simpler duality relation introduced below is as effective in the present context.

Although inspired in part by this literature, our duality theory for the monopolist's optimal pricing problem differs from the above multi-good auction optimization in several ways:
\begin{enumerate}
	\item[(a).] In the auction setting, each item can only be sold to at most one buyer, resulting in Lipschitz constraint in the single bidder problem (which becomes a nonlocal constraint on the assignment with multiple bidders), while in the nonlinear pricing model, bunching can occur in which multiple agents choose the same product.
	\item[(b).] In the auction setting, each buyer can get multiple goods, while in nonlinear pricing, each buyer would choose exactly one product which might be the outside option.
	\item[(c).] In the auction setting, the seller has no manufacturing costs, and thus, the objective functional is linear with respect to the indirect utility, whereas ours is nonlinear.
\end{enumerate}

In this paper, we specify a minimization problem that is dual to the nonlinear pricing problem over indirect utilities and prove that the primal and dual optima are both attained and their values are equal. In Section \ref{section:model}, we introduce the multidimensional nonlinear pricing problem, including the dual approach initiated by \cite{Mirrlees71} and extended to multidimensional types by \cite{RochetChone98}. Then, we present the main strong duality and attainment results in Section \ref{section:duality}. The resulting complementary slackness conditions characterize the unique optimal solution to the Rochet-Chon\'e model. In 
Section~\ref{section:applications}, we describe the analytical solution to the square version of the problem detailing three regions: the non-participation region, the bunching region, and the customization region. 
It is worth emphasizing that the bunching region we characterize as the solution to a free boundary problem is not the one described by \cite{RochetChone98}, 
which is shown to lack consistency in \cite{McCannZhang23c}, 
but instead coincides with the numerical solution of \cite{Mirebeau16}.

\section{The model}\label{section:model}

\subsection{Monopolist's problem}
A monopolist who produces and sells products aims to find the best price menu, knowing only the manufacturing cost and the distribution of consumer types. 
Let $X \subset \R^n$ denote the set of consumers 
 $Y \subset {\R^{n}}$ the set of products
~and $f:X \rightarrow [0,\infty]$ the density of consumer types. In Section \ref{section:applications}, {we} specialize to the uniform case $f(x)=1_X(x)$.

A measurable map $x \in X \mapsto (y(x), z(x)) \in Y \times \R$ of agents to (product, price) pairs is called {\it incentive compatible} if and only if $x \cdot y(x) - z(x) \ge x \cdot y(x') - z(x')$ for all $x, x' \in X$. This condition ensures agents have no incentive to hide their types when choosing products. The map is called {\it individually rational} if and only if $x \cdot y(x) - z(x) \ge x \cdot y_{\emptyset} - z_{\emptyset}$ for all $x\in X$, where $y_{\emptyset}$ and $z_{\emptyset}$ represents the outside option and its price;
for convenience we take $y_\emptyset=0$ and $z_\emptyset=0$ henceforth. 
In the context of bi-level optimization, given a price menu $v: Y \rightarrow \R$, the map $x \in X \mapsto (y(x), v(y(x)))$ is incentive compatible 
and individually rational if for each $x \in X$, $y(x)$ solves the consumer $x$'s problem of choosing the optimal product to maximize his utility $x \cdot y - v(y)$.

{Let $c: \R^n \rightarrow \R $ be the manufacturing cost, extended by setting 
\begin{equation}\label{cost extension}
 c(y) := +\infty\ \mbox{\rm for}\  y \not\in Y. 
\end{equation}
The monopolist obtains a net profit of $v(y(x))  - c(y(x))$ when consumer $x$ chooses $y(x)$.} The monopolist's problem can be formulated as follows:
\begin{flalign}\label{problem:monopolist_1}
	\begin{cases}
		\sup \Pi[y,v]: = \int_{X} \left( v(y(x))  - c(y(x)) \right) f(x) dx \quad \text{subject to}\\
		x \in X \mapsto (y(x), v(y(x))) \text{ is incentive compatible, individually rational.}\\
	\end{cases}
\end{flalign}

\medskip

{\noindent {\bf Assumption 1}. Assume $Y$ is a closed convex cone and the outside option $y_{\emptyset} = 0$ is sold at a price $0$.	Assume $X$ is a compact convex set with a nonempty interior $\Int(X)$ and $f$ is {a probability density} on $\Int(X)$ {which is bounded below by a positive constant}
    (or else is positive and lower semicontinuous on $\Int(X)$ and satisfies a Poincar\'e inequality as in Definition \ref{D:Poincare} below.)

{We focus on the cases where the cost 
\eqref{cost extension} of the outside option is no larger than its price, 
and no larger than the cost of other products. Thus, we assume, without loss of generality, that the cost of the outside option equals to its price, which is set to be zero by Assumption 1.
}
}

{\noindent {\bf Assumption 2}}. Assume $c(0) =0$, $c$ is non-negative, continuously differentiable,  strictly convex,  and $c(y) \ge a_0 |y|^2 - a_1$ holds for all $|y| \ge M$ with constants $a_0, M >0$ and $a_1\in \R$.

\bigskip
\subsection{Dual approach}

{For any function $g: S \subset \R^n \rightarrow \R$, its Legendre-Fenchel transform is a function $g^*: S^* \rightarrow \R $ defined by 
	\begin{equation}\label{Fenchel}
		g^*(x') = \sup_{x \in S} ~\langle x, x'\rangle - g(x),
	\end{equation}
	where the domain $S^* = \{x' \in \R^n: g^*(x') < \infty  \}$. Note that $g$ (and similarly $g^*$) can be extended to a function mapping from $\R^n$ to $\R \cup \{+\infty\}$ by setting $g(x) = +\infty$ whenever $x \not\in S$,  
 in which case the supremum \eqref{Fenchel} can be taken over all $x\in \R^n$.}

For any fixed price menu $v$,
define agents' {\it indirect utility} $u: X \rightarrow \R$ as {$v^*$ restricted to $X$.}
 As a supremum of linear functions, $u$ defined above is convex and thus differentiable almost everywhere by, for instance, Rademacher's theorem. Define the subdifferential of $u$ as follows. For any $x \in X$, let  
$$\partial u (x) := \{ y \in \R^n: x \cdot y - u(x) \ge x' \cdot y - u(x'), \text{ for all } x' \in X \}.$$ When $u$ is differentiable at $x$, the subdifferential of $u$ at $x$ is a singleton set containing its gradient: $\{Du(x)\}$.
Denote by $u_{\emptyset}: X \rightarrow \R$ the utility of agents from purchasing the outside option, i.e., $u_{\emptyset}(x)
{= 0}$ for {all 
} $x\in X$ { in our stylized setting.}
In the context of \cite{RochetChone98}'s model,  the following lemma and its corollary 
are well-known: facing any price menu $v$, 
they assert that a convex gradient gives the map from each consumer type to the product he selects (to maximize his utility).

\begin{lemma}[Indirect utility encodes products selected]
\label{lemma:y=Du}
For an agent $x \in X$ facing a price menu $v$, suppose his indirect utility is attained by an optimal product $y \in Y$. Then $y \in \partial v^* (x)$ (i.e., $y = Dv^*(x)$ if $v^*$ is differentiable at $x$).
\end{lemma}
\begin{proof}
	By definition of $v^*$, for any $x' \in X$, $v^*(x') \ge x' \cdot y - v(y)$. Since $y$ is an optimal choice for agent $x$, one has $v^*(x) = x \cdot y - v(y)$. Therefore, for any $x' \in X$, $x\cdot y - v^*(x) = v(y) \ge x'\cdot y - v^*(x')$. By the definition of subdifferential, one has $y \in \partial v^*(x)$. 
	
	If $v^*$ is differentiable at $x$, $\partial v^*(x) = \{D v^*(x)\}$. In this case, $y = Dv^*(x)$.
\end{proof}

As a direct consequence of the above lemma, we have the following result exhibiting the explicit dependence of agents' optimal choice on the pricing menu, which could also be obtained independently from the Envelope theorem.

\begin{corollary}
	Let $y: X \rightarrow Y$ represent the map from an agent to a product that maximizes his utility facing a price menu $v$. Then $y(x) = Dv^*(x)$ for almost every $x \in X$.
\end{corollary}
\begin{proof}
	Apply Lemma \ref{lemma:y=Du} to all the agents $x$ where $v^*$ is differentiable, then the conclusion follows from the observation that $v^*$ is differentiable almost everywhere.
\end{proof}

{In the weighted Hilbert space $H^1_f$ defined after \eqref{L^p_f} below, l}et 
\begin{equation}\label{admissable_old}
	\mathcal{U} := \left\{u\in { H^1_f} \mid u \text{ is convex}, Du(X) \subset Y, \text{ and } u \ge  u_{\emptyset}\equiv 0\right\}
\end{equation}
denote the set of admissible indirect utilities that
corresponds to the individually rational and incentive compatible (product, price) pair. 
Then $\mathcal{U}$ 
is a pointed convex cone, i.e., $\mathcal{U} \cap (-\mathcal{U}) = \{0\}$ and 
	$s_1 u_1 + s_2 u_2 \in \mathcal{U}$ for any scalars $s_1, s_2 \ge 0$ and $u_1, u_2 \in \mathcal{U}$.

It is also well-known that the problem \eqref{problem:monopolist_1} can be reformulated as the following maximization problem over indirect utilities as was done in \cite{RochetChone98}:
\begin{flalign}\label{problem:monopolist_2old}
	\sup_{u \in \mathcal{U}} \left\{\Phi[u] : = \int_{X} \left[ x \cdot Du(x) - u(x)  - c(Du(x)) \right] f(x) dx\right\}.
\end{flalign}

{In Section \ref{section:duality}}, we will show the attainment of this supremum and the uniqueness of the maximizer. Moreover, suppose $\bar u$ is the unique maximizer of $\Phi$, then $(D\bar u, {\bar u}^*)$ serves as an optimal product-price pair of the original problem \eqref{problem:monopolist_1}. Note that the maximizers of $\Pi$ in \eqref{problem:monopolist_1} may not be unique.

\subsection{Notation}\label{subsection:notation}
Here, we introduce certain function spaces equipped with integral norms and notation to prepare for the analysis in the next section.
For $p \ge 1$, and $Z$ a Hilbert space, 
let $L_f^p(X;Z)$ denote the set of $u: X \rightarrow Z$ satisfying
\begin{equation}\label{L^p_f}
\|u\|_{L_f^p(X;Z)} := \left(\int_X |u(x)|^p f(x) dx \right)^{1/p}<\infty;
\end{equation}
in case $Z=\R$ we write $L_f^p := L_f^p(X): = L_f^p(X;\R)$.
For $u \in L^1_f$ 
we define $\langle u \rangle_f : = \int_{X} u(x) f(x) dx$.  Similarly, whenever $G_1,G_2: X \mapsto \R^{n}$ yield
$G_1\cdot G_2 \in L^1_f$ we define
$\langle G_1, G_2 \rangle_{f} : = \int_{X} G_1(x)\cdot G_2(x) f(x) dx$.

{Denote by ${H^1_f} := W_f^{1,2}: = W_f^{1,2}(X; \R)$ the weighted Sobolev space of real-valued functions in $L_f^2(X)$ 
whose first order partial derivatives {are} in $L_f^2(X)$.
}

Denote by $\mX := L^2_f(X; \R^n)$ the weighted Lebesgue space of square-integrable vector fields on $X$ equipped with the inner product \(\langle \cdot, \cdot \rangle_f \)
such that elements $G_1, G_2\in \mX$ are equivalent if $G_1 = G_2$ holds $f$-a.e.. One can check that $\mX$ {and $W_f^{1,2}$} are Hilbert spaces.

{We say a statement holds for $f$-a.e. $x$ (or $f$-almost every $x$) or $fdx$-almost surely, if the subset of $X$ where the statement does not hold has measure zero under the measure whose density is $f(x) dx$. 
	 
	 If a function $u: X \rightarrow \R$ is (twice) differentiable, denote by $Du : = \nabla u := (\frac{\p u}{\p x_1}, \frac{\p u}{\p x_2})$ as its first derivative, by 
	 \begin{equation*}
	 	D^2 u = \begin{pmatrix}
	 		\frac{\p^2 u}{\p x_1^2}  & \frac{\p^2 u}{\p x_1 \p x_2} \\
	 		\frac{\p^2 u}{\p x_2 \p x_1}  & \frac{\p^2 u}{\p x_2^2}
	 	\end{pmatrix}
	 \end{equation*}
	 as its second order derivative, and by $\Delta u : = \frac{\p^2 u}{\p x_1^2} + \frac{\p^2 u}{\p x_2^2}$ as its Laplacian.
	 
	 A function $u$ is $C^{1,1}(X)$ if it is continuously differentiable and its derivative is Lipschitz. The space is equipped with a norm
	 \begin{equation*}
	 	\| u\|_{C^{1,1}(X)} : = \sup_{x \in X}  \max \left\{ | u(x)|,  \left| \frac{\p u}{\p x_1}(x)\right|,  \left|\frac{\p u}{\p x_2}(x)\right| \right\} + \max_{i = 1,2} \left\{
	 	\sup_{x, x'\in X, x \neq x'} \frac{\left|\frac{\p u}{\p x_i}(x)- \frac{\p u}{\p x_i}(x')\right|}{\|x - x'\|}
	 	 \right\}.
	 \end{equation*}
	 
	 If $\| u\|_{C^{1,1}}$ is merely bounded on compact subsets of $X$, then we say $u \in C_{loc}^{1,1}(X)$.

}

\section{Duality}\label{section:duality}

In this section, we begin by presenting a {dual infimum}, whose value coincides with the principal's profit maximization, assuming preferences are bilinear. We then {give conditions under which} that the values of the supremum and dual infimum are both attained.  From this duality and attainment, we obtain necessary and sufficient conditions which characterize the solutions of both optimization problems, and show both are attained uniquely.

By choosing a price menu (e.g. $u^*$), the principal aims to maximize her expected profits \eqref{problem:monopolist_2old}:
\begin{flalign}\label{maximization_problem}
	\Phi[u] : = \int_{X} \left( x\cdot Du(x) - u(x) - c(Du(x)) \right) f(x) dx,
\end{flalign}
among the resulting indirect utilities $u \in \mathcal{U} = \left\{u\in { H^1_f} \mid u \text{ is convex}, Du(X) \subset Y, \text{ and } u \ge  0\right\}$.

\medskip

{Observe that any maximizer to the above problem \eqref{maximization_problem} lies on the boundary of the cone $\mathcal{U}$. In fact, if a maximizer $u>0$ on $X$, then $(u - \min u) \in \mathcal{U}$ and $\Phi[u - \min u] > \Phi[u]$ contradict the assumption that $u$ is a maximizer. Suppose $X$ is strictly convex, \cite{Armstrong96} shows that there exists a positive area on $X$ where the constraint $u \ge 0$ binds at optimality. In addition, \cite{RochetChone98} shows that this area has mass $1$ under some induced measure.
	
Therefore, no single Euler-Lagrange equation governs the optimizer on the entire domain. {Rather, the Euler-Lagrange-{Karush-}Kuhn-Tucker equations take different forms on different regions within the domain,  depending on which constraints bind.} For an important example shown in Section \ref{section:applications}, we will divide $X$ into three regions, according to the rank of Hessian matrix of the optimizer and introduce specific forms of Euler-Lagrange equations on each region which characterize the solution.

}

Define
\begin{flalign}\label{D:Gamma}
	\Gamma:= \left\{G \in \mX  \,\Big|\,   \sup_{u \in \cU} \int_{X} \left(x\cdot Du(x) - u(x) -   G(x) \cdot Du(x) \right) f(x) dx \le 0  
	\right\}.
\end{flalign}
By definition, $\Gamma$ is convex and contains the identity map. {In the following main result of this section, we identify a dual minimization problem and show the attainment from both sides.}

\begin{theorem}[Strong duality and attainment]\label{thm:strong_duality} Under Assumption 1 and 2, 
\begin{flalign}\label{eqn:strong_duality_00}
	\sup\limits_{u \in \mathcal{U}} \Phi[u] = \inf\limits_{G \in \Gamma} \langle c^* \circ G \rangle_f.
\end{flalign} 
{Moreover, the primal supremum and dual infimum are both attained.}	Here $c^*$, $\mathcal U$ and $\Gamma$
are defined in \eqref{Fenchel},\eqref{admissable_old}, \eqref{D:Gamma},
{and $H^1_f$ and $\mX$ after \eqref{L^p_f}.}
\end{theorem}

An interpretation of this duality is as follows.  Compare the monopolist to a co-operative,  which is able to 
offer its members products $y \in Y$ at prices $c(y)$ given by the monopolist's costs.  The monopolist's
maximum profit coincides with the 
utility of such a co-op, minimized over all possible distributions of its membership $G_\#f$,
satisfying the strange constraint that if $G(x)$ is the true type of any agent who (irrationally) displays the anticipated behaviour of type $x$ when faced by the monopolist,  then for any price menu $u^*$ the latter proposes,
the expected direct benefit to the agents carrying out this deception (neglecting their costs) exceeds the monopolist's expected revenue.  

Before {stat}ing the strong duality result, which is deferred to Section \ref{subsection:no_duality_gap}, we first present {\it a weak duality} in which the maximization problem is bounded above by a convex minimization problem and derive {two important corollaries} on optimality conditions and uniqueness of the primal and dual problems.

\subsection{Weak Duality}

To motivate Theorem \ref{thm:strong_duality}, we present 
a weak duality in which the maximization problem is bounded above by a convex minimization problem.

\begin{proposition}[Weak Duality]\label{P:weak duality}
{With the same notation as in Theorem \ref{thm:strong_duality}},
	\begin{flalign}\label{weak duality}
	\sup\limits_{u \in \mathcal{U}} \Phi[u] \le \inf\limits_{G \in \Gamma} \langle  c^* (G) \rangle_f.
	\end{flalign} 
\end{proposition}
\begin{proof}
	For any $u \in \mathcal{U}$ and $G \in \Gamma$, one has
	\begin{flalign*}
		\Phi[u] 
		:= & \int_{X} \left(x\cdot Du(x) - u(x) -  c(Du(x)) \right) f(x) dx \\
		\le &  \int_{X} \left( G(x) \cdot Du(x)  -  c(Du(x)) \right) f(x) dx \\		
		\le &  \int_{X} c^*(G(x)) f(x) dx \\
		= & \langle c^* \circ G \rangle_f.
\end{flalign*}
	Here the first inequality is due to the definition of $\Gamma$, while the second comes from $ y\cdot y' \le c(y) + c^*(y')$ for any $y \in Y, y' \in \R^n$.
\end{proof}

{Note that the weak duality result does not rely on Assumption 1 or 2. Below is the complementary slackness result derived from the proof of weak duality.}

\begin{corollary}[Optimality condition]\label{Cor:complementary_slackness} 
{Suppose $c$ is strictly convex and continuously differentiable.} 
	Assume that $u$ and $G$ are feasible for the maximization and minimization problems \eqref{eqn:strong_duality_00} 
, respectively. Then 
	$\Phi[u] = \langle c^* \circ G \rangle_f$ 
	if and only if the following conditions hold:
	
	\begin{flalign}
		\text{1.~~} &~~G(x) = Dc(Du(x)) \text{~~holds $fdx$-almost surely}. &\\
		\text{2.~~} &~~\displaystyle \int_{X} \left(x\cdot Du(x) - u(x) -   G(x) \cdot Du(x) \right) f(x) dx = 0.&
	\end{flalign}
\end{corollary}

\begin{proof}
	These conclusions follow from the conditions under which the two inequalities used in the preceding proof 
	become equalities.
\end{proof}

We can see from the second condition in Corollary \ref{Cor:complementary_slackness} and the definition of $\Gamma$ that any optimizer of the minimization problem lies on the boundary of the constraint set $\Gamma$. {Thus, any naive attempt to characterize the primal solution by solving the dual problem using Euler-Lagrange equations and then characterizing the maximizer via optimality conditions is likely to fail.}

\begin{remark}[Necessary and sufficient conditions for optimality]
	\label{rmk:uniqueness}
	{Theorem 
 \ref{thm:strong_duality} implies $u \in \mathcal U$ is optimal if and only if there exists $G \in \Gamma$ such that 1-2 of Corollary \ref{Cor:complementary_slackness} hold;  similarly, a feasible $G$ is optimal if and only if there exists a feasible $u$ satisfying 1-2 of Corollary \ref{Cor:complementary_slackness}.} {Note that Proposition \ref{P:weak duality} alone implies that $u \in \mathcal U$ is optimal if there exists $G \in \Gamma$ such that 1-2 of Corollary \ref{Cor:complementary_slackness} hold; similarly, it also implies that a feasible $G$ is optimal if there exists a feasible $u$ satisfying 1-2 of Corollary \ref{Cor:complementary_slackness}.}
\end{remark}

	In Section \ref{section:applications}, we will use condition 1 in Corollary \ref{Cor:complementary_slackness} together with the strong duality theorem to verify that the solution we provided is indeed a maximizer. Below, we show that strong duality implies the uniqueness of the optimizers for both the primal and dual problems.

\begin{corollary}[Strong duality implies unique optimizers]\label{C:unique} Suppose $c$ is strictly convex and continuously differentiable, $X$ is convex, and $f$ is positive on $\Int(X)$. 
	If $\bu \in \mathcal{U}$ and $\bG \in \Gamma$ satisfy $\Phi[\bu] = \langle c^*\circ \bG \rangle_f$,
	then any minimizer $G \in \Gamma$ of \eqref{eqn:strong_duality_00}  satisfies $G(\cdot)= Dc(D\bu(\cdot)) = \bG (\cdot)$ $f$-a.e., while any maximizer
	$u \in \mathcal{U}$ satisfies  $u=\bar u$. 
\end{corollary}

\begin{proof}
	Since {duality} asserts $\Phi[u] \le  \langle c^* \circ G \rangle_f$ for all $u \in \mathcal{U}$ and $G \in \Gamma$,  the assumption $\Phi[\bu] = \langle c^* \circ \bG \rangle_f$ shows $\bu$ maximizes and $\bG$ minimizes.
	Any other minimizer $G \in \Gamma$ satisfies $G(x) = Dc(D\bu(x))$, $f$-almost surely, by Corollary \ref{Cor:complementary_slackness}.
	
	Similarly, any other maximizer $u \in \mathcal{U}$ satisfies $\bar G = Dc \circ Du$ hence $Du=Dc^* \circ \bar G = D\bar u$ $f$-a.e.\ (by the strict convexity of $c$).  Now $f>0$ implies $u-\bar u$ is constant {$f$-a.e.} on each connected component of $\Int(X)$. 
	Convexity of $\Int(X)$ {implies there is only one such connected component} and $\Phi[u]=\Phi[\bar u]$ impl{ies} this constant must vanish. {Therefore, $u \equiv \bu$ $f$-a.e.. Moreover, since both functions are convex, one has $u = \bu$ on $X$}.
\end{proof}

\subsection{Absence of duality gap and attainment}\label{subsection:no_duality_gap}

We shall next establish strong duality, meaning the values of the maximization and minimization problems introduced above coincide.
Let us first sketch a proof of the complementary inequality to \eqref{weak duality},  
by interchanging the order of the infimum and supremum
to find the saddle in an optimization which is (separately) linear in $u$ but convex in $G$:
\begin{flalign*}
& \sup_{u\in\mathcal{U}} \langle  x \cdot Du(x) -u(x) - c(Du(x)) \rangle_f\\
=& \sup_{u \in \mathcal{U}} \inf_{S:Y \longrightarrow \R^n_{+}}   \langle x \cdot Du(x) -u(x) - S(Du(x)) \cdot Du(x) + c^*(S(Du(x))) \rangle_f\\
 \ge &  \sup_{u \in \mathcal{U}} \inf_{G:X \longrightarrow \R^n_{+}}   \langle x \cdot Du(x) -u(x) - G(x)\cdot Du(x) + c^*(G(x)) \rangle_f\\
= & \inf_{G:X \longrightarrow \R^n_{+}}   \langle c^*(G(x)) \rangle_f +
 \sup_{u \in \mathcal{U}} \langle x\cdot Du(x) -u(x) - G(x) \cdot Du(x) \rangle_f \\
 = & \inf_{G \in \Gamma}   \langle c^* \circ G \rangle_f.
\end{flalign*}

One may apply Fenchel-Rockafellar duality to justify the above argument rigorously. 
{This is easiest to do if one assumes constants $0 < a_0 < a_0'$ exist, satisfying 
\begin{equation}\label{quadratic bounds}
a_0 {\rm\bf I}_n \le D^2 c(y) \le a_0' {\rm\bf I}_n \quad {\rm for\ all} \quad y \in Y,
\end{equation}
where ${\rm\bf I}_n$ denotes the identity matrix in $\R^{n \times n}$.
In this case}, one can show there is no duality gap in \eqref{weak duality} by applying directly classical duality results such as \cite[Corollary 16A]{Rockafellar74}. Alternatively, one can verify condition \eqref{condition_Adom} in Theorem \ref{thm:F-R duality} quoted below, then the strong duality follows. {In Appendix~\ref{app:Proofs of strong duality} we provide details of the latter argument.} 
Recall the following strong duality theorem from \cite[Theorem 4.4.3]{BorweinZhu04}.

\begin{theorem}[Fenchel-Rockafellar Duality Theorem {\cite[Theorem 4.4.3]{BorweinZhu04}}]\label{thm:F-R duality}
	Let $A$ and $B$ be Banach spaces, $\phi:  A\rightarrow \R \cup \{+\infty\}$ and $\psi: B \rightarrow \R\cup\{+\infty\}$ be convex functions, and $T: A \rightarrow B^*$ be a bounded linear map where $B^*$ is the Banach space dual to $B$. Denote by $\phi^*$ and $\psi^*$ the Legendre transforms of $\phi$ and $\psi$, respectively, and by $T^*$ the adjoint of $T$. Suppose that $\phi$, $\psi^*$ and $T$ satisfy 
	\begin{flalign}\label{condition_Adom}
	T (\dom \phi) \cap {\rm cont} \psi^* \ne \emptyset,
	\end{flalign}
	where  ${\rm cont} \psi^* \subset \dom(\psi^*)$ represents the set of all points where $\psi^*$ is finite and continuous.
	Then 
	\begin{flalign}\label{min_max}
	\inf_{x\in A} \{ \phi(x) + \psi^*(Tx)\} = \sup_{y\in B} \{-\phi^*(T^*y)- \psi(-y)\}.
	\end{flalign} 
	In addition, the supremum on the right hand side is attained if finite.
\end{theorem}

\medskip

{However, without the quadratic bounds \eqref{quadratic bounds},
the conditions for strong duality to hold, either Corollary 16A(b) in \cite{Rockafellar74} or \eqref{condition_Adom}, become tricky to verify. Readers may try, for instance, the case $c(y) = |y|^4$. Instead, in the following, we show a strategy to prove Theorem \ref{thm:strong_duality} via (i) demonstrating the strong duality theory for approximated primal and dual problems and (ii) taking the limit at both sides. }

Denote by 
\begin{align*}
\Phi_{\varepsilon}[u] 
&:= \Phi[u] - \varepsilon \|Du\|_{{L^2_f}(X;\R^n)} 
\\&= \int_{X} \left(x\cdot Du(x) - u(x) - c(Du(x))\right) f(x) dx - \varepsilon \langle |Du|^2 \rangle^{\frac{1}{2}}_f
\end{align*} and by 
\[\Gamma_{\varepsilon} : = \bigcap_{u \in \mathcal{U}} \left\{G \in \mX  \Big| \int_{X} \left(x\cdot Du(x) - u(x) -  G(x)\cdot Du(x) \right) f(x) dx \le \varepsilon \langle |Du|^2 \rangle^{\frac{1}{2}}_f \right\}.
\]
{We apply the Fenchel-Rockafellar duality theorem to a perturbed version of both problems:}

\begin{theorem}[Strong duality for perturbed problems]\label{thm:strong_duality_a}
	Let $\varepsilon >0$. {Under Assumption 1 and 2,}
	\begin{flalign}\label{eqn:strong_duality_a}
		\sup_{u \in \mathcal{U}}  \Phi_{\varepsilon}[u] = \inf_{G \in \Gamma_{\varepsilon}} \langle c^* \circ G \rangle_f.
	\end{flalign}
	{Moreover, the primal supremum and dual infimum are both attained.}	
\end{theorem}

{See Appendix \ref{app:Proofs of strong duality} for a detailed proof.  Similar to the results in Corollary \ref{Cor:complementary_slackness} and \ref{C:unique}, we have the following result for the perturbed problems.   
}

\begin{remark}[Optimality conditions and uniqueness] \label{rmk:optimality_condition} {Suppose $c$ is strictly convex and continuously differentiable, $X$ is convex, and $f$ is positive on $\Int(X)$.} 
	Assume $\bu_{\varepsilon}$ and $\bG_{\varepsilon}$ are the corresponding optimizers of the $\varepsilon$-perturbed maximization and minimization problem \eqref{eqn:strong_duality_a}, respectively. Then the same proofs as Corollaries \ref{Cor:complementary_slackness} and \ref{C:unique} yield
	\begin{itemize}
		\item $\bG_{\varepsilon}(x) = Dc(D\bu_{\varepsilon}(x))$ holds $f$-almost surely.
		\item $\int_{X} \left(x\cdot D\bu_{\varepsilon}(x) - \bu_{\varepsilon}(x) -   \bG_{\varepsilon}(x) \cdot D\bu_{\varepsilon}(x) \right) f(x) dx = \varepsilon \langle |D\bu_{\varepsilon}|^2 \rangle^{\frac{1}{2}}_f$.
		\item The optimizers in \eqref{eqn:strong_duality_a} are uniquely determined $f$-a.e.
	\end{itemize}
\end{remark}

{Below, we provide a sketch proof of Theorem \ref{thm:strong_duality} where one can see that the function spaces defined in Section \ref{subsection:notation} are essential for applying the compactness results in \cite{Carlier02} to take limits. For a detailed proof, see Appendix \ref{app:Proofs of strong duality}.
}

\medskip

{\noindent{\bf Sketch proof of Theorem \ref{thm:strong_duality}}:	For each $\varepsilon \ll 1$, denote by $\bu_{\varepsilon}$ and $\bG_{\varepsilon}$  an optimizer of each side in \eqref{eqn:strong_duality_a}, respectively. From the primal problem formulation, one can show that $\bu_{\varepsilon}$ lives in $\mathcal{U}_1$, a bounded subset  of $\mathcal{U} \cap W^{1,2}$. The compactness of $\mathcal{U}_1$ provided by \cite{Carlier02} shows that the sequence has a limit $\tilde{u}$ (up to a subsequence, same below) such that the derivatives of the sequences converge to $D\tilde u$. Then complementary slackness in Remark \ref{rmk:optimality_condition} implies $\{\bG_{\varepsilon}\}_{\varepsilon}$ converges to $Dc(D\tilde u)$. Then, we show that $Dc(D\tilde u)$ is a minimizer to the dual problem.
	
	On the other hand, the compactness of $\mathcal{U}_1$, together with upper semi-continuity of $\Phi$, implies the existence of a maximizer $\bu$ of $\Phi$ in $\mathcal{U}_1$. This implies that $\bu$ is also a maximizer of the primal problem. Taking limit of \eqref{eqn:strong_duality_a} yields \eqref{eqn:strong_duality_00}.\qed
}

\section{Application to monopolist nonlinear pricing on the square}\label{section:applications}

In this section, we apply the duality theory above to the 2D square model of \cite{RochetChone98}, whose proposed solution to this model provided {a seminal example of the optimality of product-line} bunching in multidimensions, beyond \cite{Armstrong96}'s desirability of exclusion.  More generally, Rochet and Chon\'e gave an abstract characterization of the unique optimal solution to the multidimensional analog of \cite{MussaRosen78}'s problem in terms of the existence of suitable Lagrange multipliers
whose positive and negative parts are in convex order on each bunched group of consumers.   
Unfortunately, the consequences of this characterization are delicate to work out in examples. Indeed,  some aspects of their predicted solution $\bar u$ to the 2D square model turn out not to be supported by \cite{Mirebeau16}'s subsequent numerics, which motivated us to show that the proposed solution of Rochet and Chon\'e is not self-consistent in \cite{McCannZhang23c}: it contradicts their own continuity claim for the assignment map $D \bar u:X \longrightarrow Y$, also confirmed up to the boundary of the square by \cite{CarlierLachandRobert01}.  This regularity was subsequently improved in the interior of $X$ 
by \cite{CaffarelliLions06+}, whose results combined with Carlier and Lachand-Robert's yield
\begin{equation}\label{regularity}
\bu \in C^1(X) \cap C^{1,1}_{loc}(\Int X);
\end{equation}
(similar regularity is now known to hold
in the more general setting of \cite{FigalliKimMcCann11}
by results of \cite{Chen23} and \cite{mccann2023c}).
Using our aforementioned duality along with perturbation techniques from the calculus of variations,  we go on to describe how the conjectured solution can be modified to restore its consistency with both theoretical and computational predictions, by allowing the bunched lines of consumers freedom to vary in their direction as well as their length.  This leads to a novel free boundary problem in partial differential equations, foreshadowed in \cite{McCannZhang23c}
and detailed below.

In Section~\ref{section:2d_model}, we introduce Rochet-Chon\'e's 2D square model and our proposed solution in terms of a free boundary problem which allows an overlooked form of bunching to be selected by a significant fraction of the agents.
Section~\ref{section:verification_b} 
applies the strong duality theory of Section~\ref{section:duality} above to show
 that any convex solution to this free boundary problem is indeed the unique maximizer of Rochet-Chon\'e's 2D square model. 
Section \ref{subsection:RC} contrasts our solution to the one originally proposed by \cite{RochetChone98}. 
In Section~\ref{section:interpretation}, we interpret our results economically and relate them to other recent developments.
Finally, in Section~\ref{section:sketch_optimal_solution} we derive our free boundary problem from a solution ansatz using the calculus of variations, to show that the optimal payoff $\bar u$ satisfies our free boundary problem as soon as it is consistent with the ansatz
(which is necessarily more general than that of \cite{RochetChone98}).  Although the boundary of the region separating bunching from customization is still selected by matching the values and derivatives of the solution of an ODE (in the bunching region) to a PDE (in the unbunched region),  the geometry of our bunching is more complicated than Rochet and Chon\'e proposed
and the ODE \eqref{slope E-L}--\eqref{offset E-L} which govern it derived below and announced in \cite{McCannZhang23c} are new.

\subsection{A free-boundary reformulation of Rochet-Chon\'e 2D square model}\label{section:2d_model}

For $a\ge 0$, let the square $X = [a, a+1]^2$ denote the set of consumer types equipped with density $f(x) \equiv 1$ on $X$. For each product $y \in Y = [0,\infty)^2$, let $c(y) = \frac{1}{2}|y|^2$ represent the manufacturing cost. The outside option is $(0,0) \in Y$ whose price is set to be no greater than $0$.  Thus the monopolists problem \eqref{admissable_old}--\eqref{problem:monopolist_2old} becomes
\begin{equation}\label{admissable}
\mathcal{U} = \{u\in  H^1 \mid u \text{ is convex}, Du(X) \subset [0,\infty)^2, \text{ and } u \ge  u_{\emptyset} \equiv 0\}
\end{equation}
\begin{flalign}\label{problem:monopolist_2}
		\sup_{u \in \mathcal{U}} \left\{\Phi[u] : = \int_{X=[a,a+1]^2} \left( x \cdot Du(x) - u(x)  - \frac12 |Du(x)|^2 \right) dx\right\}.
\end{flalign}
Guided by theoretical and numerical evidence,  we follow the strategy of  \cite{RochetChone98}, by making a series of ad hoc assumptions to identify a candidate optimizer $\bu$ for \eqref{admissable}--\eqref{problem:monopolist_2},  whose optimality can then be confirmed by duality (thus affirming validity of the ad hoc assumptions a posteriori).

Any convex function $u$ is twice differentiable Lebesgue a.e., hence divides almost all of $X$ into three different regions $\Omega_0$, $\Omega_1$, and $\Omega_2$, according to the rank ($0,1$ or $2$) of its Hessian matrix 
$D^2 u$.   The uniqueness of the optimal payoff $u=\bu$ established by \cite{RochetChone98} (also implied by {Corollary \ref{C:unique}}) ensures the resulting regions are symmetrical under reflection
$x_1 \leftrightarrow x_2$ through the diagonal.  Since they can be interpreted as an excluded region $\Omega_0$ of low types (where the participation constraint binds),  a bunching region $\Omega_1$ of intermediate types (where incentive compatibility, hence the convexity constraint on $\bar u$, binds),  and an unconstrained region $\Omega_2$ of high types, we shall assume they are ordered from the lower-left to the upper-right corner of the square. 
 More precisely, we assume 
there are upper semicontinuous functions $t_{i.5}:[-1,1] \longrightarrow [2a,2a+2]$ over the antidiagonal,
satisfying $t_{0.5}<t_{1.5}$, which parameterize the boundaries between these regions:
\begin{align}
\label{Omega_0}
\{\Delta \bu=0\} \subset \Omega_0
&=\{ (x_1,x_2) \in X : x_1 + x_2 \le  t_{0.5}(x_1-x_2)\} 
\\ \label{Omega_1}
\{ \det D^2 \bu =0 < \Delta \bu\} \subset \Omega_1
&= \{ (x_1,x_2) \in X :  t_{0.5}(x_1-x_2)<x_1 + x_2 \le  t_{1.5}(x_1-x_2)\}
\\ \label{Omega_2}
\{\det D^2\bu >0 \} \subset \Omega_2
&= \{ (x_1,x_2) \in X : t_{1.5}(x_1-x_2) < x_1 + x_2 \}
\end{align}
with $\Omega_i$ having connected interior and $\Omega_i \subset \cl[\Int \Omega_i]$ for each $i\in\{0,1,2\}$.
Although the geometry encoded in this assumption can probably be relaxed to account for subdomains $\Omega_i$ with boundaries parameterized in different ways, {we do not know how to relax or confirm the topology encoded in this assumption: namely that $\Omega_1$ separates $\Omega_0$ from $\Omega_2$ and that all three have connected interiors,} as suggested by \cite{Mirebeau16} and others' numerics.
 In the region $\Omega_1$, it then follows that all bunches 
$\Omega_1 \cap (D\bu)^{-1}(y)$ are given by line segments with endpoints on the boundary
of $\Omega_1$, meaning the graph of $\bar u$ is a {\em ruled surface}:
for $u \in C^2(\Int X)$ this is a classical fact, which is extended to the lower regularity \eqref{regularity} available in our context
by Cale Rankin in Lemma \ref{L:ruled surfaces} below.

So far,  our assumptions are consistent with all available theoretical and numerical evidence concerning the problem, but we shall now depart from \cite{RochetChone98},  who suppose all of the bunches in $\Omega_1$ have endpoints on $\p X$ and cross the diagonal,  and hence that 
$\bu(x)$ depends only on $t=x_1+x_2$ in $\Omega_1$.  Although the affine behaviour of $u$ in the interior of $\Omega_0$ rules out the possibility of bunches ending on 
$\p \Omega_0 \setminus \p X$,  it is perfectly plausible that some of the bunches
in $\Omega_1$ have endpoints on $\p \Omega_2$.  Therefore, inspired by \cite{Mirebeau16}'s numerics,  we allow for the possibility that some of the bunches in $\Omega_1$ have one endpoint on $\p X$ and the other on $\p \Omega_2$.    More precisely, we postulate the existence of a constant $t_{1.0} \in [t_{0.5}(s),t_{1.5}(s)]$
such that $\bar u$ depends only on $t=x_1+x_2$ in the subdomain
\begin{equation}\label{Omega_1^0}
\Omega_1^0 
:= \{ (x_1,x_2) \in \Omega_1 : x_1 + x_2 \in (t_{0.5}(x_1-x_2),t_{1.0}]\},
\end{equation}
but depends on varying convex combinations of $x_1$ and $x_2$ in the complementary ranges 
\begin{equation}\label{Omega_1^pm}
\Omega_1^\pm := \{(x_1,x_2) \in \Omega_1 \setminus \Omega_1^0  : \pm (x_1-x_2) \ge 0 \}
\end{equation}
of $\Omega_1$,  below and above the diagonal.

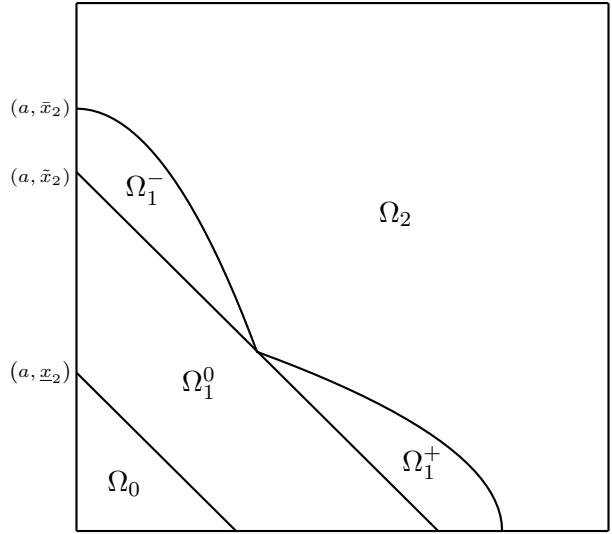
\begin{figure}[h]
	\begin{tikzpicture}[thick, scale=7]
		\draw[domain=0:1] plot (0, \x) ;
		\draw[domain=0:1] plot (\x, 1) ;
		\draw[domain=0:1] plot (1,\x) ;
		\draw[domain=0:1] plot (\x, 0) ;
		\draw[domain=0:0.34] plot (\x,  -1.2+2*sqrt{1-\x*\x*4}) ;
		\draw[domain=0:0.34] plot (-1.2+2*sqrt{1-\x*\x*4}, \x);
		\draw[domain=0:0.3] plot (\x, 0.3-\x);
		\draw[domain = 0:0.68] plot (\x, 0.68-\x);
		
		\node at (0.09, 0.09) {$\Omega_0$};
		\node at (0.23, 0.28) { $\Omega_{1}^{0}$};
		\node at (0.13, 0.65) {$\Omega_{1}^{-}$};
		\node at (0.65, 0.13) { $\Omega_{1}^{+}$};
		\node at (0.6, 0.6) {$\Omega_{2}$};
		\node at (-0.068, 0.3) {\tiny $\left(a,\underline x_2\right)$};
		\node at (-0.068, 0.67) {\tiny $\left(a,\tilde x_2\right)$};
		\node at (-0.068, 0.8) {\tiny $\left(a,\bar x_2\right)$};
	\end{tikzpicture}
	\caption{Partition of $X$}
	\label{fig:2}
\end{figure}
By $x_1 \leftrightarrow x_2$ symmetry,  it suffices to describe $\bar u$ in just one of these two regions,  say $\Omega^-_1$.  From Lemma \ref{L:ruled surfaces}, we know $\Omega_{1}^{-}$ will be foliated by line segments along which $\bu$ is affine, also called {\em isochoice sets},  {\em bunches,} or {\em leaves} of the foliation.  
It will prove convenient to parameterize the leaves of this foliation by their angle $\theta$ to the horizontal and their lengths $R(\theta)$.  The explicit formulation of our free boundary problem for the solution $\bar u$ to Rochet and Chon\'e's square example 
requires us to work out some details of this parameterization to express equations \eqref{slope BC}--\eqref{offset E-L} below
for the slope $m(\theta)$ and (unknown) left boundary value $b(\theta)$ of $\bar u$ along the leaves of this foliation in $\Omega_1^-$.

\begin{figure}[h]
	\begin{tikzpicture}[thick, scale=15]
		\draw[domain=0.68:0.8] plot (0, \x) ;
		\draw[domain=0:0.34] plot (\x,  -1.2+2*sqrt{1-\x*\x*4}) ;
		\draw[domain = 0:0.34] plot (\x, 0.68-\x);
		\draw[dashed, domain=0:0.3] plot(\x, 0.71-0.9*\x);
		
		\node at (0.05, 0.75) {$\Omega_{1}^{-}$};
		\node at (-0.039, 0.67) {\footnotesize $(a, \tilde x_2)$};
		\node at (-0.039, 0.80) {\footnotesize $(a, \bar x_2)$};
		\node at (-0.052, 0.71) {\footnotesize $(a, h(\theta))$};
		\node at (0.21, 0.55) {$\Omega_{1}^{-}(\theta)$};
	\end{tikzpicture}
	\caption{$\Omega_{1}^{-}$}
	\label{fig:3}
\end{figure}
	
	For each $x \in \Omega_{1}^{-}$,  
	let $\theta(x)$ denote the angle the line segment through $x$ makes with the horizontal
	and $r(x)$ the distance along it relative to some fixed point $(a, h(\theta))$ which 
	is the endpoint of the segment on $\partial X$.  Inverting this change of variables yields 
\begin{equation}\label{leaf parameterization}
	\bx(r,\theta) = (a, h(\theta)) + r(\cos \theta,\sin \theta),
\end{equation}
with the Jacobian of this transformation having inverse
\begin{equation}\label{Jacobian}
\frac{\p (r,\theta)}{\p (\bar x_1,\bar x_2)} =
 \left(\begin{array}{cc}
\cos\theta & -r \sin\theta
\\ \sin\theta & h' +r \cos\theta
\end{array}\right)^{-1}
=\frac{1}{h'\cos\theta + r}
\left(\begin{array}{cc}
h'+r\cos\theta & r \sin\theta
\\ -\sin\theta & \cos\theta
\end{array}\right)
\end{equation}
so that
\begin{align}
\label{area element} 
 dx_1 dx_2 
&= (r+ h'\cos\theta) dr d\theta.
\end{align}
	
	The fact that $\bu$ is affine along each such segment means there exist real functions $m(\theta)$ and $b(\theta)$
	(representing the slope and boundary value of $u$ along the segment passing through $(a,h(\theta))$ at angle $\theta$ to the horizontal), so that
\begin{equation}\label{ruled surface}
	\tilde \bu(r,\theta) := \bu(\bx(r,\theta)) =: b(\theta) + r m(\theta).
\end{equation}
	Differentiating with respect to $r$ and $\theta$ yields
	\begin{eqnarray}
		\label{sys1}     m(\theta) &=& \quad \cos\theta \frac{\p \bu}{\p x_1} (\bx(r, \theta)) + \sin\theta \frac{\p \bu}{\p x_2}(\bx(r, \theta));
		\\ \label{sys2}  m'(\theta) &=& -\sin\theta \frac{\p \bu}{\p x_1}(\bx(r, \theta)) + \cos\theta \frac{\p \bu}{\p x_2}(\bx(r, \theta));
		\\ 
\nonumber		 b'(\theta) &=& \qquad h' (\theta) \frac{\p \bu}{\p x_2}(\bx(r, \theta))
	\end{eqnarray}
while inverting \eqref{sys1}--\eqref{sys2} gives
	\begin{equation}
		 \label{syss1}     
		 D\bar u \equiv
		\left(
		\begin{array}{c}
			\frac{\p \bu}{\p x_1} (\bx(r, \theta))
			\\ \frac{\p \bu}{\p x_2}(\bx(r, \theta))
		\end{array}
		\right)
		=  \left(
		\begin{array}{cc}
			\cos\theta & -\sin\theta
			\\ \sin\theta & \cos\theta
		\end{array}
		\right)
		\left(
		\begin{array}{c}
			m(\theta)
			\\ m'(\theta)
		\end{array}
		\right).
	\end{equation}
	Therefore, $b$ must satisfy the consistency condition
	\begin{equation}\label{sys3_2}
		 b'(\theta) =  h'(\theta) \frac{\p \bu}{\p x_2}(\bx(r, \theta))  
		 = h'(\theta) [m(\theta) \sin\theta + m'(\theta) \cos\theta].
	\end{equation}
	Moreover, \eqref{syss1} also implies $ \left(\frac{\p \bu}{\p x_1}, \frac{\p \bu}{\p x_2}\right) (\bx( \cdot , \theta))$ is independent of {$r$} for each $\theta$, which coincides with the fact that all the types of consumers on this line segment $\bx(\cdot, \theta)$ would choose the same product $Du\circ \bx(\cdot, \theta)$.  On $\Omega_1^-$, combining \eqref{syss1} and \eqref{Jacobian} with the chain rule yields
\begin{equation}\label{Hessian}
D^2 \bar u(\bar x(r,\theta)) \equiv 
\begin{pmatrix}
\frac{\p^2 \bu}{\p x_1^2} &\frac{\p^2 \bu}{\p x_1 \p x_2}
\\ \frac{\p^2 \bu}{\p x_2 \p x_1} &\frac{\p^2 \bu}{\p x_2^2}
\end{pmatrix}
=\frac{m''+m}{h'\cos\theta + r}
\begin{pmatrix}
				\sin^2\theta & -\sin\theta \cos\theta \\
				-\sin\theta \cos \theta & \cos^2\theta 
			\end{pmatrix}.
\end{equation}	
	
We now construct the optimal solution $\bu=u_i$ on each set $\Omega_i$ as follows. Given 
$t_{1.0}=a+\tilde x_2$ with $\underline{h}=\tilde x_2 \in [a,a+1]$ to be determined:

\begin{enumerate}
	\item[i).] On $\Omega_{0}= \{ (x_1, x_2) \in [a,a+1]^2 : x_1 + x_2 \le a + \underline x_2 \}$ with $\underline x_2 = \frac{a+\sqrt{4a^2+6}}{3}$, one has
	\begin{flalign}\label{eqn:solution_Omega_0}
		\bu \equiv 0.
	\end{flalign}

	\item[ii).] On $\Omega_{1}^{0} = \{ (x_1, x_2) \in [a,a+1]^2 :  a + \underline x_2 \le x_1 + x_2 \le a + \tilde x_2 \}$, we have
	\begin{flalign}\label{eqn:solution_Omega_1,0}
				\bu(x_1, x_2) =& \frac{3}{8}(x_1+x_2)^2 -\frac{a}{2}(x_1+x_2) - \frac{1}{2} \ln(x_1 + x_2 -2a) + C_0,
	\end{flalign}
	where $C_0 = - \frac{2a^2 +3+2a\sqrt{4a^2+6}}{12} + \frac{1}{2}\ln\left(\frac{-2a+\sqrt{4a^2+6}}{3}\right)$.
	From \eqref{eqn:solution_Omega_1,0} we can calculate the value of $\bu$ and $D\bu$ on $	\p \Omega_{1}^0 \cap \p \Omega_1^\pm$, which gives the boundary conditions appearing on the right hand side of \eqref{slope BC} and \eqref{offset E-L} below, in view of the known regularity \eqref{regularity}.

	\item[iii).] Index each isochoice segment in $\Omega_1^-$ by its angle $\theta\in (-\frac{\pi}{4},\bar \theta]$ where $\bar \theta=\frac{\pi}{2}$ for convenience.
Let $(a,{h(\theta)})$ denote its left-hand endpoint and parameterize the segment by 
distance ${  r} \in [0,R(\theta)]$ to the point $(a,h(\theta))$. Along this segment of length $R(\theta)$ assume
\begin{equation}\label{D:m,b}
u_1^-\Big((a,h(\theta)) + { r} (\cos \theta,  \sin \theta)\Big) = {m(\theta)}{  r} + {b(\theta)}.
\end{equation}

For $\underline{h}: = \tilde x_2  \in [a,a+1]$ and  ${R}:
\left[-\frac\pi4,{\frac \pi2}\right] \to \left[0, \sqrt{2}\right)$ upper semicontinuous with $R\left(-\frac\pi4 \right) = \frac1{\sqrt2} (\underline{h} -a),$ solve

\begin{equation}\label{slope BC}
\textstyle  
m( -\frac \pi 4) = 0, \qquad m'(-\frac \pi 4) = \textstyle \sqrt{2}\ { \frac{\p \bu}{\p x_1}}\left(\frac{a + {\underline{h}}}{2}, \frac{a + {\underline{h}}}{2}\right) \qquad  \mbox{\rm such that}
\end{equation}
\begin{equation}\label{slope E-L}
({m''(\theta) + m(\theta)} - {2R(\theta)})({m'(\theta)} \sin \theta - {m(\theta)} \cos \theta +a) = \frac32 {R^2(\theta) \cos\theta}.
\end{equation}
Then set
\begin{eqnarray}
\label{D:h}
{h(\theta)} &=& {  \underline{h}} + \frac13 \int_{-\pi/4}^{\theta} (m''(\vartheta) + m(\vartheta) - 2{  R(\vartheta)}) \frac{d\vartheta}{\cos \vartheta},
\\ 
{b(\theta)} &=&  
\bu(a,{  \underline{h}})+ \int_{-\pi/4}^{\theta} (m'(\vartheta) \cos \vartheta + m(\vartheta) \sin \vartheta) h'(\vartheta) d\vartheta.
\label{offset E-L}
\end{eqnarray}
	
	Given $\bar h$ and $R$, the triple $(m, b, h)$ satisfying 
	\eqref{slope E-L}--\eqref{offset E-L} exists and is unique provided 
	$m'(\theta)\sin\theta - m(\theta)\cos\theta + a \ne 0$ and $R$ is locally Lipschitz where positive.  
	Subject to these conditions,  the shape of $\Omega_{1}^{-}$ and the value of $u_1^-$ there will be uniquely determined by any $\underline{h}$ and $R:\left(-\frac\pi4,\frac\pi2\right] \to \left[0, \sqrt{2}\right)$.  We henceforth restrict our attention to choices of $\underline{h}$ and $R$ for which the resulting set $\Omega_1^-$ 
	lies above the diagonal.  In this case $\Omega_{1}^{+}$ and the value of $\bar u = u_1^+$ on $\Omega_1^+$ are determined 
	by reflection symmetry $x_1 \leftrightarrow x_2$ across the diagonal.
	Together, $u_1^\pm$ and/or \eqref{eqn:solution_Omega_1,0} define $\bar u=u_1$ on $\Omega_1$ and provide the boundary data on 
	$\p \Omega_1 \cap \p \Omega_2$
for the Poisson equation \eqref{eqn:solution_Omega_2} below.
	\item[iv).] On $\Omega_{2} = \cl (X\setminus (\Omega_{0}\cup \Omega_1))$ where $\Omega_1 =\Omega_{1}^{0}\cup\Omega_{1}^{\pm}
$, solve:
	\begin{flalign}\label{eqn:solution_Omega_2}
	\begin{cases}
		\Delta u_2 = 3, & \text{ on } \Int(\Omega_2);\\
		(Du_2(x)-x)\cdot \vec{n}(x) = 0, & \text{ on } \p \Omega_2 \cap \p X;\\
		u_2 - u_1 = 0,  & \text{ on } 
		 \p \Omega_1 \cap \p\Omega_2.
	\end{cases}
\end{flalign}

\end{enumerate}

v). For $\underline{h}\in [a,a+1]$ and $R:\left[-\frac\pi4,\frac\pi2\right] \to \left[0, \sqrt{2}\right)$,  the mixed Dirichlet-Neumann Poisson problem \eqref{eqn:solution_Omega_2} has a unique solution $u_2$ as long as $\Omega_2$ is Lipschitz, as in 
\cite{Lieberman13}.
  We finally select $\underline{h}$ and $R:\left[-\frac\pi4,\frac\pi2\right] \to \left[0, \sqrt{2}\right)$, or equivalently the shape of $\Omega^-_1$ and hence the Lipschitz domain $\Omega_2$, by the additional requirement that $u_2$ satisfy the Neumann condition
\begin{equation}\label{eqn:solution_Omega_3}
	D(u_2-u_1) \cdot \vec{n}(x) =0, \text{ on } 
	\p \Omega_1 \cap \p\Omega_2. 
\end{equation}
This is the free boundary problem which needs to be solved;  \eqref{eqn:solution_Omega_3} is necessary for the disjointly defined functions $u_i$ on $\Omega_i$ to piece together to form $\bar u \in C^1(X)$, as required by \eqref{regularity}. 

Heuristically, the numbers of equations and unknowns coincide:  our freedom to select $\p \Omega_2 \setminus \p X$ is precisely constrained by the compatibility condition \eqref{eqn:solution_Omega_3} on it.
This suggests that the free boundary problem is neither over- nor underdetermined,
and should admit a solution: i.e. a quadruple
$(\bar h, R, u_{1}^{\pm}, u_2)$ that solves \eqref{slope E-L} -- \eqref{eqn:solution_Omega_3}, or equivalently a  triple $(\bar h,R, u)$ that solves \eqref{eqn:solution_Omega_0} -- \eqref{eqn:solution_Omega_3}.  If the resulting $u$ is admissible \eqref{admissable}, 
our next theorem shows it to be the unique optimal solution of the Rochet-Chon\'e model on the square.

\bigskip	
\subsection{Sufficiency: any convex solution of our free boundary problem is the unique optimizer} \label{section:verification_b}
The following theorem shows any solution to the free boundary problem described above which is admissible \eqref{admissable} is the unique optimal solution to the monopolist's profit maximization problem on the square. A complete proof with details is postponed to Appendix \ref{app:proof_of_T:verification_b}.
	\begin{theorem}[Free boundary solutions optimize if convex]
	\label{T:verification_b}
		If $\bu \in \mathcal{U}$ satisfies \eqref{eqn:solution_Omega_0} -- \eqref{eqn:solution_Omega_3} and both $\bar u$ and 
		$\Omega_2$ are Lipschitz, then $\bar u$ is the unique maximizer to \eqref{problem:monopolist_2}.  
	\end{theorem}

{\noindent{\bf Sketch proof of Theorem \ref{T:verification_b}}:		Our duality result, Theorem \ref{thm:strong_duality}, asserts that if
	\begin{flalign}
		\int_{X} [x\cdot Du(x) - u(x) - D\bu(x) \cdot Du(x)] dx \le 0,   \quad \text{ for all } u \in \mathcal{U}
	\end{flalign}
	with equality holding at $u = \bu$,
	then $\bu$ is the unique (Lebesgue-a.e.) maximizer of \eqref{problem:monopolist_2}.
	
	In the proof of Theorem \ref{T:verification_b}, we first present that 
	$ \int_{X} [x\cdot Du(x) - u(x) - D\bu(x) \cdot Du(x)] dx  = \eqref{term:Omega{0}+Omega{1,0}_1} + \eqref{term:Omega{0}+Omega{1,0}_2} + \eqref{term:Omega{1,1}-1} + \eqref{term:Omega{1,1}-2}$. Then it can be shown that $\eqref{term:Omega{0}+Omega{1,0}_1} + \eqref{term:Omega{0}+Omega{1,0}_2} \le 0 $ with equality holding at  $u = \bu$, and $\eqref{term:Omega{1,1}-1} + \eqref{term:Omega{1,1}-2}\le 0 $ with equality at  $u = \bu$.\qed
}	
	
\subsection{Comparison of our solution to Rochet and Chon\'e's: an overlooked market segment}
\label{subsection:RC}
		
In the preceding sections, we have established a free boundary problem corresponding to the profit maximization problem and reduced the process of characterizing the maximizer to that of verifying the existence of an admissible (i.e. convex) Lipschitz solution to this free boundary problem.  Let us now compare our proposed solution to that of Rochet and Chon\'e.
		 
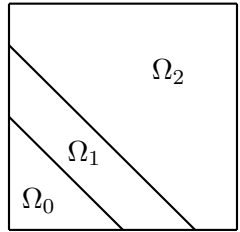
\begin{figure}[h]
	\begin{tikzpicture}[thick, scale=3]
		\draw[domain=0:1] plot (0, \x); 
		\draw[domain=0:1] plot (\x, 1); 
		\draw[domain=0:1] plot (1,\x); 
		\draw[domain=0:1] plot (\x, 0); 
		\draw[domain=0:0.8165] plot (\x,0.8165-\x); 
		\draw[domain=0:0.5] plot (\x,0.5-\x);
		
		\node at (0.13, 0.13) {$\Omega_0$};
		\node at (0.33, 0.34) {$\Omega_1$};
		\node at (0.7, 0.7) {$\Omega_2$};
	\end{tikzpicture}
	\caption{Partition of $X$ according to the rank of $D^2\bu$ given in \cite{RochetChone98}}
	\label{fig:1}
\end{figure}

As shown in Figure \ref{fig:1}, \cite{RochetChone98} claimed that the regions 
\eqref{Omega_0}--\eqref{Omega_2} where the Hessian has rank $i \in \{0,1,2\}$
are separated by two segments parallel to the anti-diagonal, so $\Omega_1 = \{(x_1,x_2) \in X : t_{0.5} <x_1 + x_2  \le  t_{1.5}\}$ 
with $t_{0.5} = \frac{4a+\sqrt{4a^2+6}}{3}$ and $t_{1.5} = 2a + \frac{\sqrt{6}}{3}$.  Thus, they do not consider the possibility of a non-empty subset $\Omega_1^\pm \subset \Omega_1$ where 
$\bu(x)$ does not just depend on $x_1+x_2$ (nor any system of equations
comparable to \eqref{D:m,b}--\eqref{offset E-L}).
Apart from that, their proposed solution is identical to ours, except that they fail to take into account that enforcing both the Dirichlet and Neumann conditions \eqref{eqn:solution_Omega_2}--\eqref{eqn:solution_Omega_3} on the line separating $\Omega_1$ from $\Omega_2$
overdetermines the Poisson problem.   We have shown their claims to lack self-consistency in \cite{McCannZhang23c}.

Motivated by \cite{RochetChone98},  different numerical approaches to variational problems with convexity constraints
have been proposed by a number of authors: 
\cite{CarlierLachand-RobertMaury01}, \cite{EkelandMoreno-Bromberg10}, \cite{Oberman13}, \cite{MerigotOudet14}, and \cite{CarlierDupuis17}. Our observation is supported by these numerics:
simulations carried out by \cite{Mirebeau16} in particular highlight that the boundary between the rank-1 and rank-2 regions of $D^2\bu$ (i.e., the boundary between $\Omega_{1}$ and $\Omega_{2}$ shown in the left picture of Figure \ref{fig:mirebeau}) is not a line segment. Moreover, in the same paper, Mirebeau also showed that the corresponding products, purchased by consumers on this boundary under the optimal pricing menu, form the non-zero curvature part of the red curve (as the boundary of the yellow/green region) in the right picture of Figure \ref{fig:mirebeau}. In the left picture of Figure \ref{fig:mirebeau}, the two ends of the boundary between $\Omega_{1}$ and $\Omega_{2}$ bend towards the anti-diagonal, providing more room for $D\bu$ to grow.   

\begin{figure}[h]
	\includegraphics{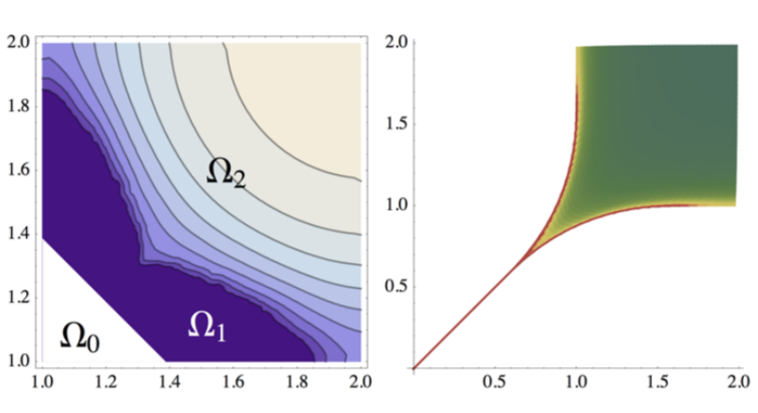}
	\caption{Numerics from \cite{Mirebeau16}. Left: level sets of $\det D^2\bu$ with $\bu = 0$ on $\Omega_{0}$ and $\det D^2 \bu = 0$ on $\Omega_{0} \cup \Omega_{1}$; Right: distribution of products sold by the monopolist.}
	\label{fig:mirebeau}
\end{figure}

\subsection{Economic interpretation and related phenomena}
\label{section:interpretation}

Let us now discuss a few aspects of our proposed solution.  Recall that the optimal indirect utility $\bar u$ is related to the optimal 
price menu $\bar v$ through the Legendre transform \eqref{Fenchel}.  More precisely, letting $u_+$ be the largest extension of $\bar u$ from $X=[0,1]^2$ to $\R^2$ which remains convex and coordinatewise non-decreasing, so that $\nabla u_+(\R^2) \subset Y=[0,\infty)^2$,
Theorem 4.6 of \cite{FigalliKimMcCann11} shows
\begin{equation}\label{price sandwich}
\bar v \ge u_+^* \quad {\rm on}\ Y = [0,\infty)^2, \qquad {\rm\ with\ equality\ on}\ \nabla \bar u(X),
\end{equation}
so $\bar v(y)=\bar u_+^*(y)$  for each product $y$ actually sold.

It is well-known that any failure of $u$ to be strictly convex at $x$ (in direction $p$) corresponds to a failure of $u^*$ to be differentiable at $y=\nabla u(x)$ (except in directions $p^\perp$) and vice versa; e.g. \cite{Rockafellar70}.
For example,  the differentiability of $u_+$ which follows from \eqref{regularity} implies that $\bar v$ coincides with
the restriction to $\nabla \bar u(X)$ of the strictly convex function $u_+^*$.  More significantly,  for any bunch $(\nabla \bar u)^{-1}(y)$
consisting of more than one point, differentiability of $u_+^*$ and hence $\bar v$ fails at $y$.  Thus,  on the red part  $\nabla \bar u(\Omega_1^0)$ of the 
diagonal depicted in Figure \ref{fig:mirebeau} (corresponding to the lower bunching region),  $\bar v$ is differentiable only in the diagonal (and not the transversal) directions.  Similarly, along the upper bunching regions $\nabla \bar u(\Omega_1^\pm)$, 
$\bar v$ cannot be extended differentiably across the boundary of $\nabla \bar u(X)$; i.e. $\bar v$ may be tangentially but not transversally differentiable along the corresponding red curves in Figure~\ref{fig:mirebeau}.  In economic terms,  if one tries to extend $\bar v$ differentiably across either of the red curves bounding $\nabla u(X)$ at $y$,  options which the monopolist does not wish to produce would be priced attractively enough to be chosen by some of the types in the bunch
$\nabla u^{-1}(y)$ and their neighbours,  an adverse selection which spoils maximality of the monopolist's profits. Alternatively: the price singularity caused by failure of the inner and outer normal derivatives of $\bar v$ to agree at those boundary points $y$ of $\nabla \bar u(X)$ where bunching occurs, leads a positive fraction ${\rm Area}[\Omega_1^\pm]$ of agents to select products on each of the red curves. 
As in \cite{ChiapporiMcCannPass17}, we expect it is possible to derive a differential equation reflecting the fact that the market must clear, by relating the local discrepancy between the inner and outer normal derivatives of  $\bar v$ (or more precisely, of $u^*_+\le \bar v$)  at such points $y$ to the one-dimensional density of products which the monopolist should produce along the red curves at $y$.

Let us also remark that in a recent investment-to-match taxation model proposed by \cite{BoermaTsyvinskiZimin22+}, simultaneously and independently of the present work,  a similar phenomenon has been numerically observed and discussed:  in their terminology and transformed coordinates, $\Omega_1$ decomposes into
a blunt bunching region $\Omega_1^0$ in which the optimal product line does not differentiate between buyers according
the sign $x_1-x_2$ distinguishing their dominant trait, as opposed to the targeted bunching regions $\Omega_1^\pm$ in which 
the optimal product line sorts along the dimension of their dominant trait and bunches in the other dimension.
In our case, the two regions can also be distinguished by the fact that the indirect utility $\bar u(x)$ is constant on each bunch in the 
blunt bunching region $\Omega_1^0$,  whereas it varies along generic bunches
in the targeted bunching region {$\Omega_1^\pm$} since \eqref{slope E-L} ensures the slope $m$ of $\bar u$ along the  
segment $(\nabla \bar u)^{-1}(y)$ cannot generally vanish. 

\subsection{Necessity: a conditional argument that the optimizer satisfies our free boundary problem}
 \label{section:sketch_optimal_solution}
  The new form of free-boundary problem whose solution we have just shown to optimize the \cite{RochetChone98} model on the square may appear mysterious.  We now motivate it by deriving the equations to be satisfied by the profit-maximizing payoff $\bu$ using perturbation arguments from the calculus of variations.  Outside $\Omega_1^\pm$ this reaffirms what was found by Rochet and Chon\'e; inside $\Omega_1^\pm$ it is new.
  However, this derivation depends 
 (a) on $\bu$ satisfying the ansatz that $X$ decomposes into three regions $\Omega_i\subset \cl[\Int \Omega_i]$ with connected interiors ordered ordered along the diagonal \eqref{Omega_0}--\eqref{Omega_2} according to the rank $i$ of $D^2\bu$; (b) that $\Omega_1$ is further subdivided  \eqref{Omega_1^0}--\eqref{Omega_1^pm},  with $\bar u$ being a function of $x_1+x_2$ on segments which foliate $\Omega_1^0$, while being affine along
 segments which start at $\p X$ and end on $\p \Omega_2$ whose slope $\tan\theta$ varies monotonically
 (and boundary intercept $h(\theta)$ increases locally uniformly) in $\Omega_1^-$; and (c) that $\Delta \bar u$  is bounded away from zero on each compact subset of $\Int(\Omega_1)$, while
  both $\Delta \bar u$ and $\det D^2 \bar u$ are bounded away from zero on each compact subset of $\Omega_2$.  These hypotheses are
 consistent with all prior theoretical and numerical results concerning the problem that we know of.

For each fixed $y\in Y$, the isochoice set
\begin{align*}
\Omega(y)&: = \{x\in X: \bu(x) = x \cdot y - \bu^*(y)\}  
\\&= \{x\in X: \bu(x) \le x \cdot y - \bu^*(y)\}  
\\&= X \cap \p \bu^*(y) 
\end{align*}
 is convex, being a level set of the convex function $x\in X \mapsto \bu(x) - x \cdot y$. On 
 $\Omega_{1}:= \Omega_{1}^{0} \cup \Omega_{1}^{-} \cup \Omega_{1}^{+}$, Lemma \ref{L:ruled surfaces} guarantees
 these isochoice sets consist of line segments whose endpoints lie on $\p\Omega_1$;
 here $\Omega_{1}^{0}$ corresponds to the region where all isochoice sets are parallel to the anti-diagonal. 
 On $\Omega_{2}$, each isochoice set is a $0$-dimensional convex set and thus a point. See Figure \ref{fig:2} for the regions in $X$ and their boundaries.

\medskip
\subsubsection{Details on $\Omega_{0}$ and $\Omega_{1}^{0}$}\label{subsubsection:Omega0_and_1,0}
In the sequel, we first have a close look at the behaviour of $\bu$ on $\Omega_0$ and $\Omega_{1}^{0}$.
The next lemma shows either the exclusion region $\Omega_0$ or its complement $X\setminus \Omega_0$ is an isosceles triangle.  See the figures below for these two possibilities. 

\begin{figure}[h]
	\centering
\begin{subfigure}{.45\textwidth}
	\centering
	\begin{tikzpicture}[thick, scale=4]
		\draw[domain=0:1] plot (0, \x) ;
		\draw[domain=0:1] plot (\x, 1) ;
		\draw[domain=0:1] plot (1,\x) ;
		\draw[domain=0:1] plot (\x, 0) ;
		\draw[domain=0:0.5] plot (\x, 0.5-\x);

		\node at (0.15, 0.15) {$\Omega_0$};
		\node at (-0.12, 0.5) {\tiny $\left(a,\underline x_2\right)$};
	\end{tikzpicture}
	\caption{ Case (i)}
	\label{fig:2.1a}
	\end{subfigure}%
\hfill
\begin{subfigure}{.45\textwidth}
	\centering
	\begin{tikzpicture}[thick, scale=4]
		\draw[domain=0:1] plot (0, \x) ;
		\draw[domain=0:1] plot (\x, 1) ;
		\draw[domain=0:1] plot (1,\x) ;
		\draw[domain=0:1] plot (\x, 0) ;
		\draw[domain=0.2:1] plot (\x, 1.2-\x);
		
		\node at (0.35, 0.35) {$\Omega_0$};
		\node at (0.2, 1.05) {\tiny $\left(\underline x_1, a+1\right)$};
	\end{tikzpicture}
	\caption{ Case (ii)}
	\label{fig:2.1b}
\end{subfigure}
	
	\caption{Shape and position of $\Omega_{0}$: two possibilities}
	\label{fig:2.1}
\end{figure}
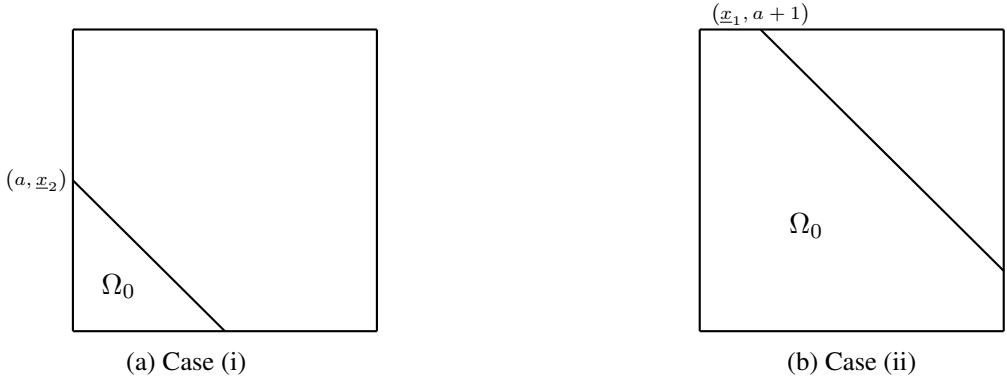

\begin{lemma}[Shape of exclusion region]\label{lemma:formula_Omega_0_shape}
	Under ansatz (a)--(c),
	$\Omega_{0}$ surrounds the lower left corner $(a,a)$ of the square $X=[a,a+1]^2$,  and either $\Omega_0$ or its complement is an isosceles triangle. 
	Moreover, $\bu \equiv 0$ on $\Omega_{0}$, and $\Omega_1^0\ne \emptyset$ in \eqref{Omega_1^0}.
\end{lemma}

\begin{proof}
The known regularity \eqref{regularity} is sufficient to ensure
$\bar u$ is affine on the (connected) interior of $\Omega_0\subset \cl[\Int \Omega_0]$, since its Hessian vanishes there.
 By symmetry,  $Du(\Omega_0)=(const,const)$;
 unless this gradient is zero,  $\max\{\bar u-\epsilon,0\}$ will generate larger profits than $\bar u$,
 hence $\bar u=0$ on $\Omega_0$ and the latter is a convex set which is reflection symmetric around the diagonal
 (by the uniqueness of optimizer asserted e.g. in {Corollary \ref{Cor:complementary_slackness}}).  
 Recall \eqref{Omega_0} implies $(a,a) \in \Omega_0$.
 Choose a point $x'=(d,d) \in \p\Omega_0 \setminus (a,a)$.
 Our ansatz (a) 
implies $x' \in \p \Omega_1$, hence is approximated by points $x_k$ on the diagonal in the interior of $\Omega_1$.
Our ansatz (b) implies
each point $x_k$ belongs to a segment in $\Omega_1$ symmetric around the diagonal with endpoints on $\p \Omega_1$ on which $\bar u$ is a non-zero constant.  For sufficiently large $k$,
the endpoints of these segments cannot lie on $\cl[\Omega_2]$ (which is disjoint from $\cl[\Omega_0]$ by hypotheses
\eqref{Omega_0}--\eqref{Omega_2}).  Nor can they lie on $\Omega_0$, since $\bu(x_k) \ne 0$.  Thus the endpoints must lie on $\p X$, hence $\Omega_1^0 \ne \emptyset$. The limit of these segments is a boundary segment of $\Omega_0$ parallel to the antidiagonal and starting and ending on $\p X$. 
\end{proof}

 From now on, we will continue the following analysis based on Case (i). The same characterization argument works equally well in Case (ii).  Lemma \ref{lemma:formula_b_Omega_0} will specify $\underline{x}_2$ and {eliminate the second case}. 

Since $\Omega_{1}^{0}$ represents the region where all equivalence classes are parallel to the anti-diagonal, by symmetry, we know $D\bu(\Omega_{1}^{0})$ is a subset of the diagonal $y_1 = y_2$ in the space $Y$ of products. Denote by $\vec{n}_1$ the unit direction parallel to the anti-diagonal of $X$. Then $\p_{\vec{n}_1} \bu = 0$ on $\Omega_{1}^{0}$. A perturbation argument on the function class where the directional derivatives along $\vec{n}_1$ vanish now yields the following lemma.

\begin{lemma}[Lower bunching region]\label{lemma:formula_Omega_1^0}
	 
	On $\Omega_{1}^{0}$, the Euler-Lagrange equation for \eqref{admissable}--\eqref{problem:monopolist_2} under ansatz (a)-(c), implies
	\begin{flalign*}
		\bu(x_1, x_2) = \frac{3}{8}(x_1+x_2)^2 -\frac{a}{2}(x_1+x_2) -  C_1 \ln(x_1 + x_2 -2a) + C_0,
	\end{flalign*}
	where $C_0 = -\frac{3}{8}(\underline x_2 + a)^2 + \frac{a}{2}(\underline x_2 + a) + C_1 \ln(\underline x_2 -a)$ and $C_1 = \frac{3}{4}(\underline x_2)^2 - \frac{1}{2}a \underline x_2 -\frac{1}{4}a^2$
	 are constants.
\end{lemma}

\begin{proof}[Proof of Lemma \ref{lemma:formula_Omega_1^0}]
	Since $\p_{\vec{n}_1} \bu = 0$ on $\Omega_{1}^{0}$, denote $\bu(x_1, x_2) := g(x_1+x_2)$.
	The hypothesized positive lower bound (c) for $g''$ ensures $g$ can be perturbed within $\mathcal U$ 
	by any smooth
	function $h(x_1 + x_2)$ on $\Omega_1^0$ 
	vanishing in a neighbourhood of $\p\Omega_{1}^{0} \cap \Int(X)$.  This perturbation yields
	\begin{flalign}\label{ODE_first_part}
			\left(2-2g''(t)\right) \left(t-2a\right) + t -2 g'(t) =0, \ \ \ \text{ on } \underline x_2 +a< t< \tilde x_2+a,
	\end{flalign}
	in the distributional sense, hence $g \in C^{1,1}_{loc}$ on this interval.
With boundary conditions $g(\underline x_2 + a) =0= g'(\underline x_2 +a)$, it is not hard to find the explicit 
formula of $g$ by solving the above ordinary differential equation (ODE).
\end{proof}

From the explicit solution above, one can see that the optimal solution is constant along the isochoice segments in $\Omega_{1}^{0}$ and that the $C^{1,1}_{loc}(\Int X)$ 
regularity 
provided by
\cite{CaffarelliLions06+} 
cannot be improved to $C^2_{loc}(\Int X)$ in any neighbourhood of the segment $\p \Omega_0 \cap \p \Omega_1$.

\medskip

\subsubsection{A verification of the Euler-Lagrange equation on $\Omega_{2}$.}\label{subsection:Omega_2}

\begin{lemma}[Customization region for top types]
\label{lemma:ODE:Omega_2}
	If $\bu \in \mathcal U$ optimizes the monopolist's profits \eqref{problem:monopolist_2} under ansatz (a)-(c), then it satisfies 
	\begin{flalign}\label{Poisson again2}
		\begin{cases}
			\Delta \bu = 3, & \text{ on } \Int(\Omega_2);\\
			(D\bu(x)-x)\cdot \vec{n}(x) = 0, & \text{ on } 
			\p X \cap \p \Omega_2.
		\end{cases}
	\end{flalign}
\end{lemma}
\begin{proof}[Proof of Lemma \ref{lemma:ODE:Omega_2}]
	Let $w$ be any smooth function supported on a compact subset of $\Omega_2$. 
	The hypothesized lower bound (c) for $D^2 \bu$ in the support of $w$ ensures $\bu+\varepsilon w \in \mathcal{U}$ for $|\varepsilon| \ll 1$. Since $\bu$ is optimal, $0 = \lim\limits_{\varepsilon \rightarrow 0} \frac{\Phi[\bu + \varepsilon w] - \Phi [\bu]}{\varepsilon}$ implies
	\begin{flalign*}
		0 = & \int_{\Omega_{2}} [(x-D\bu(x))\cdot Dw(x) -w(x)] dx\\
		= & \int_{\Omega_{2}} (\Delta \bu(x) - 3) w(x) dx + \int_{\partial \Omega_{2}} \langle x-D\bu(x), \vec{n}(x) \rangle w(x) dS(x).
	\end{flalign*}
Since $\bar u$ is convex,  its continuous differentiability \eqref{regularity} to the boundary provides sufficient regularity to justify this computation.	
	Noting that $\Omega_1 \subset \cl[\Int \Omega_1]$ and \eqref{Omega_0}--\eqref{Omega_2} give 
	$\# (\cl[\p X \cap \Omega_2]\setminus \Omega_2) = 2$,  the arbitrariness of $w$ on compact subsets of $\Omega_2$ yields the Poisson Neumann problem \eqref{Poisson again2} in both the distributional and (noting \eqref{regularity}) the pointwise a.e. senses.
\end{proof}

\medskip

\subsubsection{The ODE implied by the Euler-Lagrange equations on $\Omega_{1}^{-}$.}\label{section:polar_coordinate2}

	Recall by symmetry and uniqueness that the solution $\bu$ is symmetric across the diagonal. Without loss of generality, we may therefore focus on $\Omega_{1}^{-}$ rather than $\Omega_1^+$. Our ansatz (b) assumes $\Omega_{1}^{-}$ is foliated by line segments $\bx(\cdot, \theta)$ along which the optimizing payoff $\bu$ is affine \eqref{ruled surface}.		
	Taking $\Omega_{1}^{\pm}$ to be held fixed for the moment, we need to consider two types of functions for perturbations around this optimizer. 
	\begin{enumerate}
		\item[1.] Outer perturbations $\bu + \varepsilon w$ where $w$ is affine along the same segments \eqref{leaf parameterization}, 
		i.e.
		\begin{equation}\label{affine perturbation}
		\tilde w(r,\theta) := w(\bar x(r,\theta)) =: w_0(\theta) + r w_1(\theta).
		\end{equation}
		\item[2.] Inner perturbations: functions $\bu^{\varepsilon}$ that are affine on the segments of perturbed foliations with coordinates $\vbx(r, \theta):= (a, h(\theta) + \varepsilon\tzt(\theta)) + r (\cos\theta, \sin \theta)$.
	\end{enumerate}
	
	\begin{proposition}[Outer perturbations on upper bunching region]
	\label{P:ODE:Omega{1,1}_1}
		Let $\bu$ be the maximizer of \eqref{problem:monopolist_2}. Under ansatz (a)-(c), $\Omega_{1}^{-}$ is foliated by line segments along which $\bu$ is affine. If each fixed leaf of the foliation is parameterized by  
		$\{\bx(r, \theta) = (a, h(\theta)) + r(\cos\theta, \sin\theta): 
		 r \in [0, R(\theta)] \}$, 
		while the different leaves are parameterized by $\theta \in (-\frac{\pi}{4}, \bar{\theta}]$,
				then each segment $\Omega_{1}^{-}(\theta): = \{\bx(r, \theta): r \in [0, R(\theta)]\}$ corresponds to a length $R(\theta)$ bunch of agents who prefer the same product.
		Moreover, $\bu(\bx(r,\theta))=m(\theta)r + b(\theta)$ satisfies 	
		\begin{flalign}\label{ODE:Omega{1,1}_1}
			0= \beta(\theta) +\left(2h'(\theta) \sin\theta -  h''(\theta)\cos\theta\right)  \int_{\theta}^{\bar{\theta}} \alpha(\vartheta)   d\vartheta + \alpha(\theta) h'(\theta) \cos \theta   \hspace{1cm} \text{ on } \left(-\frac{\pi}{4}, \bar{\theta} \right),
		\end{flalign}
	where  $\alpha(\theta) : = \left( {m}(\theta) + {m}''(\theta) -3 h'(\theta)\cos\theta \right) R(\theta)  -  \frac{3}{2}  R^2(\theta) + \left(m(\theta)\cos\theta -m'(\theta)\sin \theta  - a\right)  h'(\theta)$ and
	$\beta (\theta) : =\frac{1}{2}\left( {m}(\theta) + {m}''(\theta)  -3 h'(\theta)\cos\theta \right)  R^2(\theta) - R^3(\theta)$.
	\end{proposition}

	\begin{proposition}[Inner perturbations on upper bunching region]\label{P:ODE:Omega{1,1}_2}
		Let $\bu$ be the maximizer of \eqref{problem:monopolist_2}. Under ansatz (a)-(c) with		
		$\alpha$ and $\beta$ from \eqref{ODE:Omega{1,1}_1}, $\bu$ satisfies 
		\begin{flalign}\label{ODE:Omega{1,1}_2}
			\alpha(\theta)    = 0 \hspace{2cm} \text{ on } \left(-\frac{\pi}{4}, \bar\theta \right).
		\end{flalign}
	\end{proposition}

{The proofs of Proposition \ref{P:ODE:Omega{1,1}_1} and \ref{P:ODE:Omega{1,1}_2} can be found in Appendix \ref{app:proof_of_P:ODE:Omega{1,1}_1} and \ref{app:proof_of_P:ODE:Omega{1,1}_2}, respectively.
}

Combining the conclusions from Proposition \ref{P:ODE:Omega{1,1}_1} and Proposition \ref{P:ODE:Omega{1,1}_2}, one has 
$\alpha(\theta) = 0$ and $\beta(\theta) = 0$ for any $\theta \in \left(-\frac{\pi}{4}, \bar\theta \right)$. That is to say,  the Euler-Lagrange equations on $\Omega_{1}^{-}$ imply 
\begin{flalign}\label{eqn:ODE_Omega{1}^-}
	\begin{cases}
		{m}(\theta) + {m}''(\theta)  -3 h'(\theta)\cos\theta = 2R(\theta),\\
		\left( {m}(\theta) \cos\theta-{m}'(\theta)\sin\theta - a\right)  h'(\theta) + \frac{1}{2}R^2(\theta) = 0,
	\end{cases}
	\quad \text{ on }  \theta \in \left(-\frac{\pi}{4}, \bar\theta \right),
\end{flalign}
or equivalently \eqref{slope E-L}--\eqref{D:h}.

\medskip
\subsubsection{More details on $\Omega_{0}$}
\begin{lemma}[Size of exclusion region]\label{lemma:formula_b_Omega_0}
	If $\Omega_{0}$ is connected and 
	$\Omega_{0}\cap \cl(\Omega_{2}) = \emptyset$, then $\Omega_{0}$ is a right triangle occupying the left-bottom corner of $X$ and the hypotenuse of this triangle is located on the line $\{(x_1, x_2)\in X: x_1+ x_2 = a + \underline x_2\}$ with $\underline x_2 = \frac{a + \sqrt{4a^2 +6}}{3}$. Moreover, $\bu \equiv 0$ on $\Omega_{0}$.
\end{lemma}

\begin{proof}[Proof of Lemma \ref{lemma:formula_b_Omega_0}]
Based on the results provided in Lemma \ref{lemma:formula_Omega_0_shape}, we only need to show that $\Omega_{0}$ is a right triangle with $\underline x_2 = \frac{a + \sqrt{4a^2 +6}}{3}$.
	
(i). Assume $\Omega_{0}$ is a right triangle.
	Integrating the ODEs and PDEs obtained from perturbation arguments on $X\setminus \Omega_{0}$ as shown in Lemma \ref{lemma:formula_Omega_1^0}, \ref{lemma:ODE:Omega_2}, and Proposition \ref{P:ODE:Omega{1,1}_1}, one has
	$$0 = \int_{X\setminus \Omega_{0}} (\Delta \bu -3) dx + \int_{\partial X \setminus \p \Omega_{0}} \langle
	x - D\bu(x), \vec{n}(x)\rangle dS(x).
	$$
	On the one hand, for any smooth function $w$ on $X$, one has
	\begin{flalign*}
		\Phi[\bu + w] - \Phi[\bu] = &\int_{X} \left(Dw \cdot x - w - \frac{1}{2}|Dw|^2 - D\bu \cdot Dw \right) dx\\
		= & \int_{X} \left(w \cdot (\Delta \bu -3) -\frac{1}{2}|Dw|^2 \right) dx +\int_{\partial X}w \langle x-D\bu, \vec{n}(x)\rangle dS(x).
	\end{flalign*}
	
	Taking $w \equiv 1$, it follows that 
	\begin{flalign*}
		\Phi[\bu + 1] - \Phi[\bu] =  & \int_{X}  (\Delta \bu -3) dx +\int_{\partial X} \langle x-D\bu, \vec{n}(x)\rangle dS(x)\\ 
		= &	\int_{\Omega_{0}}  (\Delta \bu -3) dx +\int_{\partial X \cap \p \Omega_{0}} \langle x-D\bu, \vec{n}(x)\rangle dS(x)\\
		= &	\int_{\Omega_{0}}  -3 dx +\int_{\partial X \cap \p \Omega_{0}} \langle x, \vec{n}(x)\rangle dS(x)\\
		= & -3 |\Omega_{0}|  - a |\p X \cap \p \Omega_{0}|
	\\ = &-\frac32 (\underline x_2 -a)^2 - 2 a (\underline x_2 -a).
	\end{flalign*}
	
	On the other hand, $\Phi[\bu + 1] - \Phi[\bu] = \int_{X} -1 dx = -1$.  The resulting quadratic equation  
	$$0= 3(x-a)^2 + 4a(x-a) -2 
	$$
	implies $\underline x_2 = \frac{a + \sqrt{4a^2 +6}}{3} $.
	
	(ii). Assume  $\Omega_{0}$ is an irregular pentagon described in Figure \ref{fig:2.1b} with one side located on the line $\{(x_1, x_2)\in X: x_1+ x_2 = a + 1+ \underline x_1\}$. Following the same calculations as above, one will get $\underline x_1 = a+1$ and thus $\Omega_{0} = X$. This implies that the optimal solution $\bu \equiv 0$  on $X$. However, one can check that $\Phi[\bu] = 0$ is not optimal because $u (x_1, x_2) := \frac{2 \sqrt{2a+0.01}}{9} (x_1 + x_2)^{\frac{3}{2}} \in \mathcal{U}$ and $\Phi[u ] >0 = \Phi[\bu]$.
	
	(iii). Therefore,  $\Omega_{0}$ is a right triangle described in Figure \ref{fig:2.1a} with $\underline x_2 = \frac{a + \sqrt{4a^2 +6}}{3}$. \qedhere
\end{proof}

Plugging the value of $\underline x_2$ into the explicit formula of $\bar{u}$ on $\Omega_{1}^0$, one obtains the following corollary.

\begin{corollary}
	In Lemma \ref{lemma:formula_Omega_1^0}, $C_0 = - \frac{2a^2 +3+2a\sqrt{4a^2+6}}{12} + \frac{1}{2}\ln\left(\frac{-2a+\sqrt{4a^2+6}}{3}\right)$ and $C_1 = \frac{1}{2}$.
\end{corollary}

\bigskip

	\section{Conclusion and future work}

This paper establishes a strong duality with attainment for the monopolist problem with bilinear preferences. We apply this duality theory to analyze the Rochet-Chon\'e 
bidimensional square model with quadratic costs. This leads to a free boundary problem that requires identifying the unique domain boundary for which the solution of a new ODE (describing a targeted bunching region in which the isochoice segments rotate) on one side of the boundary can be differentiably matched to the solution of a Poisson Neumann problem that characterizes the optimal payoff on the other side of the boundary.  Under an ansatz more general than Rochet and Chon\'e's,  we show that solving this free boundary problem is both necessary and sufficient for optimality.   We show each bunch corresponds to a price
gradient discontinuity across the boundary of the optimal product line.
It remains a challenging open problem to give a rigorous proof either that this free boundary problem admits an admissible (i.e.~convex) solution or, alternately, that the optimal payoff satisfies the hypotheses of our necessity ansatz.

{We close this paper by conjecturing the existence of a convex solution to the free boundary problem.  This conjecture is consistent with all 
theoretical and numerical evidence concerning the problem that we are aware of.  We hope to tackle this conjecture in the future,
perhaps using a fixed-point or dynamical flow argument, or a variational principle.}

\appendix

\section{Proof of {Theorems
\ref{thm:strong_duality} and \ref{thm:strong_duality_a}:} Strong duality}\label{app:Proofs of strong duality}

This section will show the proof of Theorem \ref{thm:strong_duality} using an approximating strategy. The idea is to prove the strong duality for perturbed problems (Theorem \ref{thm:strong_duality_a}) before taking the limit to zero. However, its proof {is more} straightforward under the simplifying hypothesis \eqref{quadratic bounds}, as {we first show.
Recall that a probability density $f$ on a convex set $X\subset\R^n$ satisfies a Poincar\'e inequality if bounded away from zero:

\begin{definition}[Poincar\'e inequality]
\label{D:Poincare}
 We say the {\it Poincar\'e inequality} holds with weight $f \ge 0$ on $X$, if there exists a constant $C_f>0$ such that $u \in L^1_{loc}(X)$ and $Du \in L^2_f(X;\R^n)$ (defined just after \eqref{L^p_f}) imply $u \in L^2_f$ and
\begin{flalign}
	\left\langle (u - \langle u\rangle_f)^2\right\rangle_f \le C_f \left\langle |D u|^2 \right\rangle_f.
\end{flalign}
\end{definition}
}

\begin{proof}[Proof of Theorem \ref{thm:strong_duality} {assuming \eqref{quadratic bounds}}]
	
		1. Define $\psi(G) := \langle c^* \circ G \rangle_f$ for any $G\in \mX$. For each $u \in {W_f^{1,2}}$, define
		\begin{flalign*}
			\phi(u) 
			:= \begin{cases}
				\int_{X} \left( -x \cdot Du(x) + u(x)\right)f(x)dx, & u \in \mathcal{U}\\
				+\infty, & u \notin \mathcal{U}.
			\end{cases}
		\end{flalign*}
		It is easy to see that both $\psi$ and $\phi$ are convex. Define a linear mapping $T: {W_f^{1,2}} \rightarrow \mX^* $ such that $Tu = Du$ for any $u \in  W_f^{1,2}$.
		From  the definition, we know $T$ is bounded and the dual map $T^*: \mX \rightarrow \left(W_f^{1,2}\right)^*$ satisfies  $\langle T^*G, u\rangle_{W_f^{1,2}} = \langle Tu, G\rangle_{\mX} = \langle Du, G \rangle_f$ for any $u\in {W_f^{1,2}}$ and any $G\in \mX$. Thus, for any $\tilde G\in \mX ^*$, 
		\begin{flalign}\label{eqn:phi^*_00}
			\begin{aligned}
				\psi^*(\tilde G) & = \sup_{G\in \mX} \langle G, \tilde G\rangle_f - \psi(G)\\
				& =  \sup_{G\in \mX} \int_{X} G(x) \cdot \tilde G(x) f(x) dx - \int_{X} c^*(G(x)) f(x) dx\\
				& =  \int_{X} c(\tilde G(x)) f(x) dx.
			\end{aligned}
		\end{flalign}
		
		Therefore, $\psi^*(Tu) =  \langle c\circ Du\rangle_f$, for any $u \in {W_f^{1,2}}$.
		
		For any $G \in \mX$, one has
		\begin{flalign*}
			\phi^*(-T^*G) 
			&= \sup_{u \in {W_f^{1,2}}} \langle u, -T^*G \rangle_{W_f^{1,2}} - \phi(u) 
			\\ &= \sup_{u \in \mathcal{U}} \int_{X} \left(- G(x) \cdot Du(x) + x \cdot Du(x) - u(x) \right)f(x) dx \\
			&=\begin{cases}
				0, & \text{if } G \in \Gamma; 
				\\
				+\infty, &\text{otherwise}. 
			\end{cases}
		\end{flalign*}
		
		For any fixed $G \in \mX^*$ and any $\varepsilon>0$, there exists $\delta = \min\left\{1, \frac{\varepsilon}{\left\|Dc\circ G \right\|_{L^2_f} + \frac{a_0'}{2}}\right\}$, for any $\tilde G \in \mX^*$ with $ \left\|\tilde G - G \right\|_{L^2_f} \le \delta$, one has
		
		\begin{flalign*}
			|\psi^*(\tilde G) - \psi^*(G)| = & \left|\int_{X} \left[c(\tilde G(x)) - c(G(x)) \right] f(x) dx \right| \\
			\le & \int_{X} \left|c(\tilde G(x)) - c(G(x)) \right|  f(x) dx \\
			\le & \int_{X} \left[ \left|Dc(G(x))\right| \left|\tilde G(x) - G(x) \right|  + \frac{a_0'}{2} \left|\tilde G(x) - G(x) \right|^2 \right] f(x) dx \\
			\le & \left\|Dc\circ G \right\|_{L^2_f} \left\|\tilde G - G \right\|_{L^2_f} + \frac{a_0'}{2} \left\|\tilde G - G \right\|_{L^2_f}^2 \\
			\le &  \left\|Dc\circ G \right\|_{L^2_f} \delta + \frac{a_0'}{2} \delta^2 \le \varepsilon.
		\end{flalign*}
		Here, the second inequality comes from the assumption that $D^2 c \le a_0' {\rm\bf I}_n$, and the third comes from the Cauchy-Schwarz inequality. Therefore, $\psi^*$ is continuous whenever it is finite. And Hypothesis \eqref{condition_Adom} is satisfied. For instance, one can check that $x \in T (\dom \phi) \cap {\rm cont} \psi^*$. 
		
		Hence, the Fenchel-Rockafellar Duality Theorem implies
		\begin{flalign}\label{eqn:strong_duality_varepsilon_00}
			\begin{aligned}
				-  \sup_{u \in \mathcal{U}}  \Phi[u] =  \inf_{u \in \mathcal{U}}  -\Phi[u] & =	\inf_{ u\in {W_f^{1,2}} } \{ \phi(u) + \psi^*(Tu)\} \\
				& = \max_{G \in \mX} \{-\phi^*(-T^*G)- \psi(G)\} = \max_{G \in \Gamma} - \langle c^* \circ G \rangle_f = - \min_{G \in \Gamma}  \langle c^* \circ G \rangle_f .
			\end{aligned}
		\end{flalign}

		2. It remains to show that the first supremum in \eqref{eqn:strong_duality_varepsilon_00} is achieved.
		Let $\{u_{m}\}_{m=1}^{\infty}$ be a sequence in $\mathcal{U}$ such that $\lim_{m\rightarrow \infty} \Phi[u_{m}] = \sup_{u \in \mathcal{U}} \Phi[u]$. It is clear that there exists a constant $C_1 > 0$ such that $  \| Du_m \|_{L^2_f}  \le C_1$ and $ \| u_m \|_{L^1_f}  \le C_1$ for all $m$, since otherwise $\lim\sup_{m\rightarrow \infty} \| Du_m \|_{L^2_f}   = + \infty $  or  $\lim\sup_{m\rightarrow \infty} \| u_m \|_{L^1_f}   = + \infty $ and thus for $m$ large enough
		\begin{flalign*}
			-1 = \Phi[0] - 1 \le \Phi[u_{m}] =& \int_{X} \left( x\cdot Du_m(x) - u_m(x) - c(Du_m(x)) \right) f(x) dx \\
			\le & \int_{X} \left( x\cdot Du_m(x) - u_m(x) - Dc(0) \cdot Du_m(x) - a_0 |Du_m(x)|^2 \right) f(x) dx\\
			\le & \| x- Dc(0)\|_{L^2_f}  \| Du_m(x) \|_{L^2_f} - a_0 \| Du_m(x) \|_{L^2_f}^2 - \| u_m(x) \|_{L^1_f}  
		\end{flalign*}
		implies  $-1 \le \lim\inf_{m \rightarrow \infty} \left( \| x- Dc(0)\|_{L^2_f}  \| Du_m(x) \|_{L^2_f} - a_0 \| Du_m(x) \|_{L^2_f}^2 - \| u_m(x) \|_{L^1_f}  \right) = -\infty$, which is a contradiction. Here, we applied $c(y)\ge c(0) + Dc(0) \cdot y + a_0 |y|^2$  in the second inequality and the Cauchy-Schwarz inequality in the third.
		
		Since $f$ is bounded below by a positive constant (i.e., there exists $C_2 >0$ such that $f \ge C_2 >0$), there exists a constant $C_3 >0$ such that $\| Du_m \|_{L^2}  \le C_3$ and $ \| u_m \|_{L^1}  \le C_3$ for all $m$. The  Poincar\'e inequality {with uniform probability density (Definition \ref{D:Poincare} with weight $\equiv 1$)}
  then implies that  $\{u_m\}_m$ is bounded in $L^2(X)$ and thus $\{u_m\}_m$ is bounded in $W^{1,2}(X)$. Applying \cite[Proposition 1]{Carlier02} yields that there exists a convex function $\bu$ in $W^{1,2}(X)$ and a subsequence of $\{u_{m_k}\}_{k}$ such that $\{u_{m_k}\}_{k}$ converges to $\bu$ uniformly on compact subsets of $\Int(X)$ and $\{Du_{m_k}\}_{k}$ converges to $D\bu$ pointwisely almost everywhere on $\Int(X)$. Moreover, one can show that $\bu \in \mathcal{U}$. Therefore, the upper semi-continuity of $\Phi$ implies $\bu$ is a maximizer of the primal problem.\qedhere

\end{proof}

\medskip

Denote by ${\dot H^1_f} := \dot H^1_f(X; \R)$ the weighted homogeneous Sobolev space of real-valued functions on $X$ equipped with the inner product
\[
\langle u,v \rangle_{\dot H^1_f} := \int_X  Du(x) \cdot Dv(x) f(x) dx,
\]
so that elements $u,v \in {\dot H^1_f}$ are identified if $u-v =constant$ on $\Int(X)$. 

Under a similar argument in the proof of Theorem \ref{thm:strong_duality} {under the simplifying hypothesis \eqref{quadratic bounds},} one can see that the optimizers of $\Phi_{\varepsilon}$ and $\Phi$ sit in $L_f^1 \cap \dot H_f^1$, which, under the mild assumption on $f$, implies the optimizers {lie} in a bounded set of $W^{1,2}$. This allows us to use  $\dot H_f^1$, instead of $W_f^{1,2}$, as the space of utility functions. For the sake of simplicity, we re{state the definition} of $\mathcal{U}$ here: 
\begin{equation*}\label{admissable_old_dot_H1}
	\mathcal{U} := \left\{u\in {\dot H^1_f} \mid u \text{ is convex}, Du(X) \subset Y, \text{ and } u \ge  u_{\emptyset}\equiv 0\right\}.
\end{equation*}

\begin{proof}[Proof of Theorem~\ref{thm:strong_duality_a}]
	
	1. Define $\phi(G) := \langle c^* \circ G \rangle_f$ for any $G\in \mX$. For each $u \in {\dot H^1_f}$, define
	\begin{flalign*}
		\psi_{\varepsilon}(u) 
		:= \begin{cases}
			\int_{X} \left( x \cdot Du(x) - u(x)\right)f(x)dx + \varepsilon \langle |Du|^2 \rangle^{\frac{1}{2}}_f, & u \in -\mathcal{U}\\
			+\infty, & u \notin -\mathcal{U}.
		\end{cases}
	\end{flalign*}
	It is easy to see that both $\phi$ and $\psi_{\varepsilon}$ are convex. Define a linear mapping $T: \mX \rightarrow {(\dot H^1_f)}^*$ such that 
	\begin{flalign*}
		\forall u \in {\dot H^1_f},  \quad \langle u, TG\rangle_{\dot H^1_f} = \int_{X} G(x) \cdot Du(x) f(x) dx.
	\end{flalign*}
	From  the definition, we know $T$ is bounded and the dual map $T^*: {\dot H^1_f} \rightarrow \mX^*$ satisfies  $\langle T^*u, G\rangle_f = \langle u, TG\rangle_{{\dot H^1_f}}$ for any $u\in {\dot H^1_f}$ and any $G\in \mX$. Thus, for any $u \in {\dot H^1_f}$, 
	\begin{flalign}\label{eqn:phi^*}
		\begin{aligned}
			\phi^*(T^*u) & = \sup_{G\in \mX} \langle G, T^*u\rangle_f - \phi(G)\\
			& =  \sup_{G\in \mX} \int_{X} G(x) \cdot Du(x) f(x) dx - \int_{X} c^*(G(x)) f(x) dx\\
			& =  \int_{X} c(Du(x)) f(x) dx.
		\end{aligned}
	\end{flalign} 
	
	For any $G \in \mX$, one has
	\begin{flalign*}
		(\psi_{\varepsilon})^*(TG) 
		&= \sup_{u \in {\dot H^1_f}} \langle u, TG \rangle_{{\dot H^1_f}} - \psi_{\varepsilon}(u) 
		\\ &= \sup_{u \in \mathcal{U}} \int_{X} \left(- G(x) \cdot Du(x) + x \cdot Du(x) - u(x) \right)f(x) dx - \varepsilon \langle |Du|^2 \rangle^{\frac{1}{2}}_f \\
		&=\begin{cases}
			0, & \text{if } G \in \Gamma_{\varepsilon}; 
			\\
			+\infty, &\text{otherwise}. 
		\end{cases}
	\end{flalign*}
		
	Then Hypothesis \eqref{condition_Adom} is satisfied since $Tx\in  	T (\dom \phi) \cap {\rm cont} (\psi_{\varepsilon})^*$. Hence, the Fenchel-Rockafellar Duality Theorem implies
	\begin{flalign}\label{eqn:strong_duality_varepsilon}
		\inf_{G \in \Gamma_{\varepsilon}} \langle c^* \circ G \rangle_f = \inf_{G \in \mX} \{\phi(G) + (\psi_{\varepsilon})^*(TG)\} = \max_{u\in {\dot H^1_f}} \{-\phi^*(T^*u)- \psi_{\varepsilon}(-u)\} = \max_{u \in \mathcal{U}}  \Phi_{\varepsilon}[u].
	\end{flalign}

	2. It remains to show that the first infimum in \eqref{eqn:strong_duality_varepsilon} is achieved.
	
	Let $G_{0}$ be the identity map on $X$, $B: = \{G\in \mX| \langle |G|^2 \rangle_f \le \langle |G_0|^2 \rangle_f\}$ and $\tilde{\Gamma}_{\varepsilon}
	: = \Gamma_{\varepsilon} \cap B$. It is clear that $\tilde{\Gamma}_{\varepsilon} \neq \emptyset$ since it contains $G_{0}$. 
	Because $\Gamma_{\varepsilon}$ is closed, $\tilde{\Gamma}_{\varepsilon}$ is weakly compact as is $B$, which is implied by the Banach-Alaoglu theorem. 
	In addition, the existence of this minimization problem follows from the lower semi-continuity of $\langle c^*(\cdot)\rangle_f$ under the same topology.

	{\bf Claim:} $\langle c^*(\cdot)\rangle_f$ is lower semi-continuous under weak topology. 
	
	{\bf Proof}.
	For any sequence $\{G_i\}_{i = 1}^{\infty}$ on $\Gamma_{\varepsilon}$ and $G_{\infty} \in \Gamma_{\varepsilon}$ such that $G_i \stackrel{w}{\rightharpoonup} G_{\infty}$, the convexity of $c^*$ implies
	\begin{flalign*}
		c^*(G_i(x)) - c^*(G_{\infty}(x))\ge Dc^*(G_{\infty}(x))(G_i(x) - G_{\infty}(x)), \forall x\in X.
	\end{flalign*}
	Thus, 
	\[
	\langle c^* \circ G_i \rangle_f - \langle c^* \circ G_{\infty} \rangle_f
	\ge \langle Dc^* \circ G_{\infty}, G_i-G_{\infty}\rangle_f. 
	\]
	Therefore,
	\[
	\liminf\limits_{i\rightarrow \infty} \langle c^* \circ G_i \rangle_f - \langle c^* \circ G_{\infty} \rangle_f
	\ge \liminf\limits_{i\rightarrow \infty} \langle Dc^* \circ G_{\infty} , G_i-G_{\infty}\rangle_f
	=   0.\qedhere
	\]
	
\end{proof}

\medskip

\begin{proof}[Proof of Theorem~\ref{thm:strong_duality}]
	
	For each $\varepsilon \ll 1$, denote by $\bu_{\varepsilon}$ and $\bG_{\varepsilon}$  an optimizer of each side in \eqref{eqn:strong_duality_a}, respectively. It is clear that there 
	exists constant $C_1>0$ such that $\langle |D\bu_{\varepsilon}|^2\rangle_f \le C_1$ {and  $\langle \bu_{\varepsilon}\rangle_f \le C_1$} for all 
	$\varepsilon \ll 1$, since otherwise $\limsup_{\varepsilon \rightarrow 0^{+}}\langle |D\bu_{\varepsilon}|^2\rangle_f = +\infty$ {or $\limsup_{\varepsilon \rightarrow 0^{+}}\langle \bu_{\varepsilon}\rangle_f = +\infty$ } and thus
	\begin{flalign*}
		0 = \Phi_{\varepsilon}[0]  \le \Phi_{\varepsilon}[\bu_{\varepsilon}] 
		\le - \langle c(D\bu_{\varepsilon})\rangle_f + \langle |x|^2\rangle_f^{\frac{1}{2}} \langle |D\bu_{\varepsilon}|^2\rangle_f^{\frac{1}{2}}  - \langle \bu_{\varepsilon}\rangle_f
	\end{flalign*}
	implies $ 0 \le \liminf_{\varepsilon \rightarrow 0^{+}} \left(- \langle c(D\bu_{\varepsilon})\rangle_f + \langle |x|^2\rangle_f^{\frac{1}{2}} \langle |D\bu_{\varepsilon}|^2\rangle_f^{\frac{1}{2}} { - \langle \bu_{\varepsilon}\rangle_f}\right) = -\infty$, 
	which is a contradiction. Here we use the assumption that $c$ is bounded below by some parabola: $c(y) \ge a_0 |y|^2 - a_1$ holds for all $|y| \ge M$ with constants $a_0, M >0$ and $a_1\in \R$. 

	\medskip
	
	1. {Let $\mathcal{U}_1 := \{u \in  \mathcal{U}: \langle |Du|^2\rangle_f \le C_1, \langle u\rangle_f \le C_1\}$. } For all $\varepsilon \ll 1$, 
	\begin{flalign*}
		\sup_{u \in \mathcal{U}_1} \Phi[u] 
		&\ge  \max_{u \in \mathcal{U}_1} \left\{ \Phi[u] - \varepsilon \langle |Du|^2\rangle^{\frac{1}{2}}_f \right\} 
		\\&\ge \sup_{u \in \mathcal{U}_1} \left\{ \Phi[u] - \varepsilon C^{\frac{1}{2}}_1\right\} 
		\\&=  \sup_{u \in \mathcal{U}_1} \Phi[u] - \varepsilon C^{\frac{1}{2}}_1.
	\end{flalign*}
	This implies 
	\begin{flalign*}
		\sup_{u \in \mathcal{U}_1} \Phi[u] 
		&= \lim_{\varepsilon \rightarrow 0}  \max_{u \in \mathcal{U}_1} \left\{ \Phi[u] - \varepsilon \langle |Du|^2\rangle^{\frac{1}{2}}_f \right\} 
		\\&= \lim_{\varepsilon \rightarrow 0}  \max_{u \in \mathcal{U} } \left\{ \Phi[u] - \varepsilon \langle |Du|^2\rangle^{\frac{1}{2}}_f \right\}= \lim_{\varepsilon \rightarrow 0}  \max_{u \in \mathcal{U} } \Phi_{\varepsilon}[u].
	\end{flalign*}
	
	{ Since $f$ is bounded below by a positive constant, i.e., there exits $C_2 >0$ such that $f \ge C_2$. Thus, together with Poincar\'e inequality, there exists $C_3, C_4 >0$ such that $\mathcal{U}_1 \subseteq \mathcal{U} \cap \{ u \in \dot H^{1}: \|Du\|_{L^{2}} \le C_3, \|u\|_{L^{1}} \le C_3 \} \subseteq \mathcal{U} \cap \{ u \in W^{1,2}: \|u\|_{W^{1,2}} \le C_4 \}$.}
	The compactness properties of \(\mathcal{U} \cap \{ u \in W^{1,2}: \|u\|_{W^{1,2}} \le C_4 \}\) described in \cite{Carlier02} combine with { the closeness of $\mathcal{U}_1$ and }
	the upper semi-continuity of $\Phi$ to imply the existence of a maximizer. 
	Let $\bar{u}$ be a maximizer {of $\sup_{u \in \mathcal{U}_1} \Phi[u]$}.

	Suppose that 
	\begin{flalign*}
		\sup_{u \in \mathcal{U}} \Phi[u] > \max_{u \in \mathcal{U}_1} \Phi[u].
	\end{flalign*}
	Then there exists $u_1 \in \mathcal{U}$ 
	and  $\varepsilon >0$ such that \(\Phi[u_1] - \varepsilon \langle |Du_1|^2\rangle^{\frac{1}{2}}_f > \Phi[\bar{u}] = \max_{u \in \mathcal{U}_1} \Phi[u]\). Thus,
	\begin{flalign*}
		\Phi[\bu_{\varepsilon}] - \varepsilon \langle |D\bu_{\varepsilon}|^2\rangle^{\frac{1}{2}}_f  
		\ge \Phi[u_1] - \varepsilon \langle |Du_1|^2\rangle^{\frac{1}{2}}_f >  \Phi[\bar{u}] = \max_{u \in \mathcal{U}_1} \Phi[u] \ge \Phi[\bu_{\varepsilon}].
	\end{flalign*}
	This is a contradiction. Thus, 
	\begin{flalign*}
		\sup_{u \in \mathcal{U}} \Phi[u] = \max_{u \in \mathcal{U}_1} \Phi[u].
	\end{flalign*}
	Moreover, $\bar{u}$ is also a maximizer of $\Phi[u]$ in $\mathcal{U}$ and
	\begin{flalign*}
		\max_{u \in \mathcal{U}} \Phi[u] = \lim_{\varepsilon \rightarrow 0}  \max_{u \in \mathcal{U} } \Phi_{\varepsilon}[u] .
	\end{flalign*}

	2. {On the other hand, since $\{\bu_{\varepsilon}\}_{\varepsilon \ll 1} \subseteq \mathcal{U} \cap \{ u \in W^{1,2}: \|u\|_{W^{1,2}} \le C_4 \}$, 	by \cite[Proposition~1]{Carlier02}, }
	there exists a convex function $\tilde u$ and a subsequence $\{\bu_{\varepsilon_{k}}\}_{k}$ such that 
	$\{D\bu_{\varepsilon_{k}}\}_{k}$ converges to $D\tilde u$ pointwisely outside a set of zero volume.

	Recall that, from the complementary slackness (Remark \ref{rmk:optimality_condition}), $\bG_{\varepsilon} = Dc(D\bu_{\varepsilon})$ holds $fdx$-almost surely. Therefore, $\{\bG_{\varepsilon_{k}}\}_k$ converges to $\bG: = Dc(D\tilde u)$ $fdx$-almost surely.  Moreover, since $ \Gamma_{\varepsilon_k}$ is closed under the weak topology, $\bG \in \Gamma_{\varepsilon_{k}}$ for any $k>0$ and thus $\bG \in \Gamma$. Therefore, 
	\begin{flalign*}
		\liminf\limits_{k\rightarrow +\infty} \min_{G \in \Gamma_{\varepsilon_k}} \langle c^* \circ G \rangle_f 
		=  \liminf\limits_{k\rightarrow +\infty} \langle c^* \circ \bG_{\varepsilon_k} \rangle_f 
		\ge \langle c^* \circ \bG \rangle_f 
		\ge \inf_{G \in \Gamma} \langle c^* \circ G \rangle_f 
		\ge \min_{G \in \Gamma_{\varepsilon_k}} \langle c^* \circ G \rangle_f.
	\end{flalign*}
	This implies, 
	\begin{flalign*}
		\liminf\limits_{k\rightarrow +\infty} \min_{G \in \Gamma_{\varepsilon_k}} \langle c^* \circ G \rangle_f
		=  \min_{G \in \Gamma} \langle c^* \circ G \rangle_f,
	\end{flalign*}
	and $\bG \in \argmin_{G \in \Gamma} \langle c^* \circ G \rangle_f$.
	
	3. Taking limits of \eqref{eqn:strong_duality_a} yields
	\begin{flalign*}
		&&\max_{u \in \mathcal{U}} \Phi[u] = \min_{G \in \Gamma} \langle c^* \circ G \rangle_f.
		&&	\qedhere
	\end{flalign*}
\end{proof}

\section{Proof of Theorem \ref{T:verification_b}: Free boundary solutions optimize if convex}\label{app:proof_of_T:verification_b}

\begin{proof}[Proof of Theorem \ref{T:verification_b}]
	Our duality result, Theorem \ref{thm:strong_duality}, asserts that if
	\begin{itemize}
		\item[(i)] $D\bu \in \Gamma$ from \eqref{D:Gamma} and 
		\item[(ii)] 
		$\Phi [\bu] := \int_X [ x\cdot D\bu(x) - \bu(x) - \frac12 |D\bu(x)|^2] dx = \langle \frac12 |D\bu|^2 \rangle_{f=1}$,
	\end{itemize}
	then $\bu$ is the unique (Lebesgue-a.e.) maximizer of \eqref{problem:monopolist_2}.
	
	Thus, it is sufficient to show (i) that 
	\begin{flalign}
		\int_{X} [x\cdot Du(x) - u(x) - D\bu(x) \cdot Du(x)] dx \le 0,   \quad \text{ for all } u \in \mathcal{U}
	\end{flalign}
	and (ii) that equality holds at $u = \bu$. Let us remark that any convex Lipschitz function has a distributional Hessian which is a matrix-valued measure on $X$ of finite total mass.  This provides sufficient regularity to just the necessary integrations by parts.
	
	Using \eqref{leaf parameterization} to define $\tilde U(r,\theta) := u(\bx_1(r, \theta),\bar x_2(r,\theta))+u(\bx_2({r},\theta),\bx_1(r,\theta))$, and the Lipschitz continuity of the convex function $\bar u$ and $\Omega_2$ to integrate by parts, combining the area element
	\eqref{area element} with expressions 
	for the gradient \eqref{syss1} and Laplacian \eqref{Hessian} of $\bar u$ in $\Omega_1^-$ we find
	\begin{flalign}
		&\int_{X} [x\cdot Du(x) - u(x) - D\bu(x) \cdot Du(x) ]dx\\
		= & \int_{X} (\Delta \bu(x) - 3) u(x) dx + \int_{\partial X} \langle x-D\bu(x), \vec{n}(x) \rangle u(x) dS(x)\\
		\label{term:Omega{0}+Omega{1,0}_1}	= & -3 \int_{\Omega_0} u(x) dx - \int_{
			\p \Omega_0 \cap \p X}
		a u(x) dS(x)
		+ \int_{\Omega_{1}^{0}}  (\Delta \bu(x) - 3) u(x) dx  
		\\ \label{term:Omega{0}+Omega{1,0}_2}		
		& + \int_{\p \Omega_1 \cap \{x_2=a\}} \left(\frac{\partial \bu}{\partial x_2}-a\right) u(x_1,a) dx_1 +\int_{\p \Omega_1 \cap \{x_1=a\}} 
	\left(\frac{\partial \bu}{\partial x_1}-a\right) u(a, x_2) dx_2 
	\\
	\label{term:Omega{1,1}-1}		& + \int_{-\frac{\pi}{4}}^{\bar{\theta}} \int_{0}^{R(\theta)} \left[m(\theta) + {m}''(\theta) - 3(h'(\theta)\cos\theta +r)\right] 
	\tilde U(r,\theta) dr d\theta \\
	\label{term:Omega{1,1}-2}		& + \int_{-\frac{\pi}{4}}^{\bar{\theta}} \left(m(\theta)\cos\theta - m'(\theta) \sin\theta -a \right) 
	\tilde U(0,\theta) h'(\theta) d\theta
\end{flalign}
where \eqref{eqn:solution_Omega_2}--\eqref{eqn:solution_Omega_3} have been used to show that the contributions from $\bar u = u_2$ on $\Omega_2$ and its boundary are cancelled by the boundary contributions of $\bar u=u_1$ on $\p \Omega_1 \cap \p \Omega_2$. Here we may take $\bar \theta = \frac\pi2$ or $\bar \theta = \sup\{\theta \in [-\frac\pi4,\frac\pi2] \mid R(\theta)>0\}$.

From the explicit formula of $\bu$ on $\Omega_{0}\cup \Omega_{1}^{0}$, it is not hard to see that $\eqref{term:Omega{0}+Omega{1,0}_1} + \eqref{term:Omega{0}+Omega{1,0}_2} \le 0$ for any $u \in \mathcal{U}$ where equality holds for $u = \bu$. To complete the proof, without loss of generality, we only need to show $\eqref{term:Omega{1,1}-1} + \eqref{term:Omega{1,1}-2}\le 0$ for all $u \in \mathcal{U}$ and equality holds for $u = \bu$.

Since $h$ satisfies \eqref{D:h}, it follows that
\begin{flalign*}
	\eqref{term:Omega{1,1}-1} = & \int_{-\frac{\pi}{4}}^{\bar{\theta}} \int_{0}^{R(\theta)} \left(2R(\theta) - 3r\right) \tilde U(r, \theta)) dr d\theta\\
	= & \int_{-\frac{\pi}{4}}^{\bar{\theta}} \int_{0}^{R(\theta)}  
	\left(-2R(\theta)r + \frac{3r^2}{2}\right) \frac{\partial \tilde U}{\partial r}(r, \theta) 
	dr d\theta  
	+ \int_{-\frac{\pi}{4}}^{\bar{\theta}} \frac{1}{2}R^2(\theta) \tilde U(R(\theta), \theta)) d \theta.
\end{flalign*}
Using the fact that $m$ also satisfies \eqref{slope E-L}  we deduce
\begin{flalign*}
	\eqref{term:Omega{1,1}-2} = -\int_{-\frac{\pi}{4}}^{\bar{\theta}} \frac{1}{2}R^2(\theta) \tilde U(0, \theta) d\theta.
\end{flalign*}

Thus, 
\begin{flalign*}
	&\eqref{term:Omega{1,1}-1} + \eqref{term:Omega{1,1}-2} \\
	= &\int_{-\frac{\pi}{4}}^{\bar{\theta}} \int_{0}^{R(\theta)}  
	\left(-2R(\theta)r + \frac{3r^2}{2}\right) \frac{\partial \tilde U}{\partial r}(r, \theta) 
	dr d\theta + \int_{-\frac{\pi}{4}}^{\bar{\theta}} 
	\frac{1}{2}R^2(\theta)  (\tilde U(R(\theta), \theta) - \tilde U(0, \theta)) d \theta \\
	= & \int_{-\frac{\pi}{4}}^{\bar{\theta}} \int_{0}^{R(\theta)}  
	\left(\frac{1}{2}R^2(\theta)  - 2R(\theta)r + \frac{3r^2}{2}\right) \frac{\partial \tilde U}{\partial r}(r, \theta) dr d\theta.  
\end{flalign*}

Denote $\zeta(r, \theta) : = \frac{1}{2}R^2(\theta)  - 2R(\theta)r + \frac{3r^2}{2}$. Then, for each $\theta \in [-\frac{\pi}{4}, \bar{\theta}]$, one has 
$(R(\theta)-3r)\zeta(r,\theta)\ge 0$ so $\zeta$ changes sign at $r=\frac13R(\theta)$. Moreover
$\int_{0}^{R(\theta)} \zeta(r, \theta) dr = 0$.
Since $u$ is convex, we know $\frac{\partial \tilde U}{\partial r}(\cdot, \theta)$ is increasing for each fixed $\theta$. This implies 
\begin{flalign*}
	\zeta (r, \theta) \frac{\partial \tilde U}{\partial r}(r, \theta) \le \zeta (r, \theta) \frac{\partial \tilde U}{\partial r}\left( \frac{R(\theta)}{3}, \theta \right), \quad \text{ for all }	(\theta,r) \in [-\frac{\pi}{4}, \bar{\theta}] \times [0, R(\theta)].
\end{flalign*}

Therefore,
\begin{flalign*}
	\eqref{term:Omega{1,1}-1} + \eqref{term:Omega{1,1}-2} 
	= & \int_{-\frac{\pi}{4}}^{\bar{\theta}} \int_{0}^{R(\theta)}   \zeta (r, \theta) \frac{\partial \tilde U}{\partial r}(r, \theta) dr d\theta \\		
	\le & \int_{-\frac{\pi}{4}}^{\bar{\theta}} \int_{0}^{R(\theta)}  \zeta (r, \theta) \frac{\partial \tilde U}{\partial r}\left( \frac{R(\theta)}{3}, \theta \right) dr d\theta \\
	= &\int_{-\frac{\pi}{4}}^{\bar{\theta}} \frac{\partial \tilde U}{\partial r}\left( \frac{R(\theta)}{3}, \theta \right) \int_{0}^{R(\theta)}  \zeta (r, \theta)  dr d\theta 
	=  0.
\end{flalign*}

Note that when $u=\bu$, for any $\theta \in [-\frac{\pi}{4}, \bar{\theta}]$, one has $\frac{\partial \tilde U}{\partial r}( \cdot, \theta) = 2m(\theta)$ and  $\int_{0}^{R(\theta)} \zeta(r, \theta) dr = 0$. In this case,
\begin{flalign*}
	\eqref{term:Omega{1,1}-1} + \eqref{term:Omega{1,1}-2} 
	= 2\int_{-\frac{\pi}{4}}^{\bar{\theta}} \int_{0}^{R(\theta)}  \zeta (r, \theta) m(\theta) dr d\theta   
	=  2\int_{-\frac{\pi}{4}}^{\bar{\theta}}   m(\theta) \int_{0}^{R(\theta)} \zeta (r, \theta) dr d\theta  
	=  0
\end{flalign*}
as desired.
\end{proof}

\section{Proof of Proposition \ref{P:ODE:Omega{1,1}_1}: Outer perturbations on upper bunching region}\label{app:proof_of_P:ODE:Omega{1,1}_1}

	\begin{proof}[Proof of Proposition \ref{P:ODE:Omega{1,1}_1}]
	
	The regularity \eqref{regularity} known for $\bar u$ combines with Lemma \ref{L:ruled surfaces} to give the foliation of 
	$\Omega_1$.  Suppose the foliation can be parameterized by \eqref{leaf parameterization} in $\Omega_1^-$, with
	$r \in [0, R(\theta)]$, where all the leaves of the foliation intersect $\partial X$ and each leaf corresponds to a line segment at an angle $\theta \in (-\frac{\pi}{4}, \bar{\theta}]$ to the horizontal.  		
	Now consider perturbations $\bu + \varepsilon w$ of $\bu$ which are affine \eqref{affine perturbation} 
	along the same segments. 
	In the interior of $R^{-1}((0,\infty))$, we are free to prescribe any ($C^2$ smooth) $w_1(\theta)$,  but, similarly to \eqref{sys3_2}, the choice of $w_1(\theta)$ 
	determines $w_0(\theta)$ up to an additive constant:
	\begin{equation}
		w_0'(\theta) =  h'(\theta) \frac{\p w}{\p x_2}(\bx(r, \theta))  = h'(\theta) [\sin\theta w_1(\theta) + \cos\theta w_1'(\theta)],
	\end{equation}
	i.e., apart from an additive constant of integration,  $w_0$ is the linear image of $w_1$ under a particular integro-differential operator
	(and in fact depends bilinearly on $h(\theta)$ and $w_1$). 
	
	Assume $w \equiv 0$ in $\Omega_0\cup \Omega_{1}^{0} \cup \Omega_{1}^{+}$. One can easily check that $\bu +\varepsilon w$ stays non-negative with non-negative partial derivatives for $|\varepsilon| \ll 1$. 
	Analogously to \eqref{Hessian} we compute
	\begin{equation}
		D^2 \bar w(\bar x(r,\theta)) := 
		\begin{pmatrix}
			\frac{\p^2 w}{\p x_1^2} &\frac{\p^2 w}{\p x_1 \p x_2}
			\\ \frac{\p^2 w}{\p x_2 \p x_1} &\frac{\p^2 w}{\p x_2^2}
		\end{pmatrix}
		=\frac{w_1''(\theta)+w_1(\theta)}{h'\cos\theta + r}
		\begin{pmatrix}
			\sin^2\theta & -\sin\theta \cos\theta \\
			-\sin\theta \cos \theta & \cos^2\theta 
		\end{pmatrix}.
	\end{equation}	
	Since $D^2u$ and $D^2w$ are multiples of the same rank-one matrix (namely, projection orthogonal to the bunch), 
	we see	$\det(D^2\bu + \varepsilon D^2 w)=0$ 
	on $\Omega_{1}^{-}$. Moreover, $\bu + \varepsilon w$ inherits a positive Laplacian (c) from $u$ hence remains in $\mathcal{U}$ for $|\varepsilon| \ll 1$. 
	
	Now compute the Euler-Lagrange equation satisfied by $\bu$ using an  arbitrary perturbation $w_1$ 
	(and the corresponding $w_0$).  Since $\bu$ is optimal,  using the area element
	\eqref{area element} and expressions 
	for the gradient \eqref{syss1} and Laplacian \eqref{Hessian} of $\bar u$ in $\Omega_1^-$, from
	$0 = \lim\limits_{\varepsilon \rightarrow 0} \frac{\Phi[\bu + \varepsilon w] - \Phi [\bu]}{\varepsilon}$ 
	and \eqref{Poisson again2} on $\Omega_{2}$ we deduce
	\small 
	\begin{flalign*}
		0 =& \int_{\Omega_{1}^{-} \cup \Omega_2} \left[(x-D\bu) \cdot D w - w\right] dx\\
		= &  \int_{\Omega_{1}^{-}} (\Delta \bu(x) -3) w(x) dx + \int_{
			\p X \cap \p \Omega_1^-} \langle x-D\bu(x), \vec{n}(x)\rangle w(x) dS(x)\\
		= & \int_{-\frac{\pi}{4}}^{\bar{\theta}} \int_{0}^{R(\theta)} \left( \frac{m(\theta) + {m}''(\theta)}{h'(\theta)\cos\theta +r}-3\right) (w_0(\theta)+rw_1(\theta))(h'(\theta)\cos\theta +r)dr d\theta\\
		& + \int_{-\frac{\pi}{4}}^{\bar{\theta}} \left(m(\theta) \cos \theta -m'(\theta) \sin \theta - a\right) w_0(\theta) h'(\theta) d\theta\\
		= & \int_{-\frac{\pi}{4}}^{\bar{\theta}} 
		[\alpha(\theta)w_0(\theta) + \beta(\theta)w_1(\theta)] d\theta
	\end{flalign*}
	\normalsize
	where $\alpha(\theta) : = \left( {m}(\theta) + {m}''(\theta) -3 h'(\theta)\cos\theta \right) R(\theta)  -  \frac{3}{2}  R^2(\theta) + \left(\cos\theta {m}(\theta) -\sin \theta {m}'(\theta) - a\right)  h'(\theta)$ and $\beta (\theta) : =\frac{1}{2}\left( {m}(\theta) + {m}''(\theta)  -3 h'(\theta)\cos\theta \right)  R^2(\theta) - R^3(\theta)$.
	Once again, this integration by parts can be justified since the convexity of the Lipschitz function \eqref{regularity} implies its Hessian is a matrix-valued Radon measure.

	Note that $w_0(-\frac{\pi}{4}) = w_1(-\frac{\pi}{4}) = 0$ and $w_0'(\theta)  = h'(\theta) (\sin\theta w_1(\theta) + \cos\theta w_1'(\theta)) $. Thus,
	\begin{flalign*}\label{eqn:ODE_integration_by_parts}
		\begin{aligned}
			0 = & \int_{-\frac{\pi}{4}}^{\bar{\theta}} \alpha(\theta) w_0 (\theta) + \beta(\theta) w_1(\theta) d \theta \\
			= & \int_{-\frac{\pi}{4}}^{\bar{\theta}} \alpha(\theta) \int_{-\frac{\pi}{4}}^{\theta}    h'(\vartheta) (\sin\vartheta w_1(\vartheta) + \cos\vartheta w_1'(\vartheta))  d\vartheta  + \beta(\theta) w_1(\theta)d \theta \\	  
			= & \int_{-\frac{\pi}{4}}^{\bar{\theta}} \left\{\alpha(\theta) \left[\int_{-\frac{\pi}{4}}^{\theta}    \left(2h'(\vartheta) \sin\vartheta -  h''(\vartheta)\cos\vartheta \right) w_1(\vartheta)   d\vartheta\right]  + \alpha(\theta) h'(\theta) \cos \theta w_1(\theta) + \beta(\theta) w_1(\theta)\right\}d \theta \\
			= & \int_{-\frac{\pi}{4}}^{\bar{\theta}}    \left\{\left(2h'(\theta) \sin\theta -  h''(\theta)\cos\theta\right) w_1(\theta)  \left[\int_{\theta}^{\bar{\theta}} \alpha(\vartheta)   d\vartheta\right]  + \alpha(\theta) h'(\theta) \cos \theta w_1(\theta) + \beta(\theta) w_1(\theta)\right\}d \theta \\
			= & \int_{-\frac{\pi}{4}}^{\bar{\theta}}   \left\{ \left(2h'(\theta) \sin\theta -  h''(\theta)\cos\theta\right)  
			\left[\int_{\theta}^{\bar{\theta}} \alpha(\vartheta)   d\vartheta \right]+ \alpha(\theta) h'(\theta) \cos \theta  + \beta(\theta) \right\} w_1(\theta)d \theta. 
		\end{aligned}
	\end{flalign*}
	The above equality holds for any smooth $w_1$ on $[-\frac{\pi}{4}, \bar{\theta}]$ such that $w_1(-\frac{\pi}{4})=0$. This implies
	\begin{flalign*}
		\left(2h'(\theta) \sin\theta -  h''(\theta)\cos\theta\right)  \int_{\theta}^{\bar{\theta}} \alpha(\vartheta)   d\vartheta + \alpha(\theta) h'(\theta) \cos \theta  + \beta(\theta) =0 \hspace{1cm} \text{ on } (-\frac{\pi}{4}, \bar{\theta}).\qedhere
	\end{flalign*}
\end{proof}

\section{Proof of Proposition \ref{P:ODE:Omega{1,1}_2}: Inner perturbations on upper bunching region}\label{app:proof_of_P:ODE:Omega{1,1}_2}

	\begin{proof}[Proof of Proposition \ref{P:ODE:Omega{1,1}_2}]
	Let $\vbzt(\theta) : = h(\theta) + \varepsilon\tzt(\theta)$ where $\tzt: [-\frac{\pi}{4}, \bar\theta] \rightarrow \R$ such that $\tzt(-\frac{\pi}{4}) =\tzt(\bar \theta) =0$ and $(\vbzt)' (\theta)>0$ for any $\theta \in [-\frac{\pi}{4}, \bar\theta]$.
	
	Consider perturbations $\bu^{\varepsilon}$ that  are affine on the segments of perturbed foliations with coordinates $\vbx(r, \theta) := (a, \vbzt(\theta)) + r (\cos\theta, \sin \theta)$ such that 
	\begin{flalign*}
		\widetilde{\vbu}(r, \theta): = \vbu(\vbx(r, \theta)) := \vbuz(\theta) + r m(\theta),
	\end{flalign*} 
	where $\vbuz$ is determined by $m$ and $\vbzt$ as in \eqref{sys3_2}:
	\begin{equation*}
		(\vbuz)'(\theta) 		=  (\vbzt)'(\theta)  [m(\theta)\sin\theta + m'(\theta)\cos\theta]
	\end{equation*}
	with $\vbuz(-\frac{\pi}{4}) = b(-\frac{\pi}{4})$.
	Moreover, as in \eqref{sys1}--\eqref{syss1} we have
	\begin{flalign*}
		m(\theta) &= \cos\theta \frac{\partial \vbu}{\partial x_1}(\vbx(r, \theta)) + \sin\theta \frac{\partial \vbu}{\partial x_2}(\vbx(r, \theta)),\\
		m'(\theta) &= -\sin\theta \frac{\partial \vbu}{\partial x_1}(\vbx(r, \theta)) + \cos\theta \frac{\partial \vbu}{\partial x_2}(\vbx(r, \theta)), \\
		\frac{\partial \vbu}{\partial x_1}(\vbx(r, \theta)) &= \cos\theta m(\theta) - \sin\theta m'(\theta) 
		=\frac{\partial \bu}{\partial x_1}(\bx(r, \theta)), \\
		\frac{\partial \vbu}{\partial x_2}(\vbx(r, \theta)) &= \sin\theta m(\theta) + \cos\theta m'(\theta) 
		=\frac{\partial \bu}{\partial x_2}(\bx(r, \theta)). 
	\end{flalign*}
	In particular, $\left(\frac{\partial \vbu}{\partial x_1}, 	\frac{\partial \vbu}{\partial x_2}\right)(\vbx(\cdot, \theta))$ is constant for each $\theta$, i.e., all the types of consumers on this line segment $\vbx(\cdot, \theta)$ would prefer the same product $D\vbu \circ \vbx(\cdot, \theta)$ over all the other products.
		
	Noting that $(\vbu, D\vbu) = (\bu, D\bu)$ on $\p \Omega_1^-\cap \p \Omega_1^0$, $\frac{\p \bu}{\p x_1} = \frac{\p \bu}{\p x_2}$ at $(\frac{a+\tilde x_2}{2}, \frac{a+\tilde x_2}{2})$, and $D\vbu = D\bu$ at $(a, \bar x_2) = \bx(0, \bar{\theta}) = \vbx(0, \bar{\theta})$, we extend $\vbu$ to $X$ such that 
	\begin{enumerate}
		\item[1.] $\vbu \equiv \bu$ on $\Omega_{0} \cup \Omega_{1}^{0}$;
		\item[2.] $\vbu$ on $\Omega_{1}^{\pm}$ is symmetric along diagonal;
		\item[3.] $\vbu$ convex on $\Omega_{2}$ with 
		$\frac{\p \vbu}{\p x_i}\mid_{\p \Omega_2 \cap \{x_i=a\}} =a$ for\footnote{From Lemma \ref{lemma:ODE:Omega_2}, we know $\bu$ satisfies 
			$\frac{\p \bu}{\p x_i}\mid_{\p \Omega_2 \cap \{x_i=a\}} =a$ for $i=1,2$.} 
		$i=1,2$ and $ \int_{\Omega_{2}}  |D\left[\vbu(x)-\bu(x)\right]|^2 dx = o(\varepsilon)$.
	\end{enumerate} 
	From the equation $D\vbu \circ \vbx(r, \theta) = D\bu \circ \bx(r, \theta)$, we know that $D\vbu$ inherits the correct sign from $D\bu$. Since $\vbu = \bu >0$ on $
	\p \Omega_1^0 \cap \p \Omega_1^-$, the fundamental theorem of calculus yields $\vbu \ge 0$ on $\Omega_{1}^{-}$. Taking additional derivatives of $D\vbu$ yields $D^2 \vbu = \frac{r+h'(\theta)\cos\theta}{r+(\vbzt)'(\theta)\cos\theta}D^2 \bu$, so $\vbu$ inherits convexity from $\bu$ on $\Omega_{1}^{-}$. By symmetry, $\vbu, D\vbu$, and $D^2 \vbu$ also have the correct signs on $\Omega_{1}^{+}$.
	Since $\frac{\p \vbu}{\p x_1}\mid_{\p\Omega_2 \cap (\p \Omega_1 \cup \{x_1=a\})} 
	\ge 0$, the convexity of $\vbu$ on $\Omega_{2}$ implies $\frac{\p \vbu}{\p x_1} \ge 0$ on $\Omega_{2}$. Similarly, $\frac{\p \vbu}{\p x_2} \ge 0$ holds on $\Omega_{2}$.  In addition, $D\vbu \ge 0$ implies $\vbu \ge 0$ on $\Omega_{2}$ since $\vbu \ge 0$ on $\p \Omega_1 \cap \p \Omega_2$.
	Thus, $\vbu$ remains in $\mathcal{U}$ for  $|\varepsilon| \ll 1$.

	Now, let's compute the Euler-Lagrange equation satisfied by $\bu$ using perturbations $\vbu$.
	
	\begin{flalign}
		& \Phi[\vbu]- \Phi[\bu] +  \frac{1}{2} \int_{\Omega_{1}^{\pm}\cup \Omega_{2}}  |D\vbu-D\bu|^2 dx\\
\nonumber		= & \int_{\Omega_{1}^{\pm}\cup \Omega_{2}} \bigg(x\cdot D\left(\vbu-\bu\right) -\left(\vbu-\bu\right) - \langle D\vbu-D\bu, D\bu\rangle  
		\bigg) dx\\
\nonumber		= & \int_{\Omega_{1}^{\pm}\cup \Omega_{2}} \left(\vbu-\bu\right)(\Delta \bu -3) dx 
		+ \int_{((\p \Omega_1^\pm \cup \Omega_2)\cap \p X) \cup (\p \Omega_1^\pm \cap \p \Omega_1^0) } 
		\left(\vbu-\bu\right) \langle x-D\bu(x), \vec{n}(x)\rangle dS(x) \\
	\label{term:D2}	= & 2 \int_{\Omega_{1}^{-}}  \left(\vbu-\bu\right)(\Delta \bu -3) dx \\
	\label{term:D3}		&+ 2 \int_{\p \Omega_1^{-}\cap \p X}  \left(\vbu-\bu\right) \langle x-D\bu(x), \vec{n}(x)\rangle dS(x) 
	\end{flalign}
	
	For any $\theta \in [-\frac{\pi}{4}, \bar\theta]$ and $r\in [0, R(\theta)]$, define $\vlegr(r, \theta) \in [0, \infty)$ and $\veth(r, \theta)\in [-\frac{\pi}{4}, \bar\theta] $ such that $\vbx(\vlegr (r, \theta), \veth(r, \theta)) := \bx(r, \theta)$. Note that $\vlegr$ and $\veth$ are well defined since $(\vbzt)' >0$.
	
	From the definition of $\vlegr$, $\veth$ and $\vbzt$, 
	\begin{flalign*}
		\begin{cases}
			\vlegr(r, \theta) \cos \veth(r, \theta) = r\cos \theta; \\
			\vbzt(\veth(r, \theta)) + \vlegr(r, \theta) \sin \veth(r, \theta) = h(\theta) + r\sin\theta
		\end{cases}
	\end{flalign*}
	while for $|\varepsilon| \ll 1$:
	\begin{flalign*}
		\begin{cases}
			\vlegr(r, \theta) = r - \varepsilon \frac{r \tzt(\theta) \sin \theta}{h'(\theta) \cos\theta +r} + o(\varepsilon),\\
			\veth(r, \theta) = \theta - \varepsilon \frac{\tzt(\theta) \cos \theta}{h'(\theta) \cos\theta +r} + o(\varepsilon).
		\end{cases}
	\end{flalign*}
	
	Taylor expanding yields
	\begin{flalign*}
		& \vbu(\bx(r, \theta)) - \bu(\bx(r, \theta)) \\
		= & \vbu(\vbx(\vlegr(r, \theta), \veth(r, \theta)))
		- \bu(\bx(r, \theta)) \\
		= & \vbuz(\veth(r, \theta)) - b(\theta)  + \vlegr(r, \theta) m(\veth(r, \theta) )- r m(\theta)\\
		= & \vbuz(\theta) - b(\theta) - \varepsilon \frac{\tzt(\theta)}{h'(\theta) \cos\theta +r} \left( (\vbuz)'(\theta)  \cos\theta+ m'(\theta)  r\cos\theta+ m(\theta)  r\sin\theta\right) + o (\varepsilon)\\
		= &   \int_{-\frac{\pi}{4}}^{\theta} (\vbuz)'(\vartheta) - (b)'(\vartheta) d \vartheta 
		- \varepsilon \frac{\tzt(\theta)}{h'(\theta)}(b)'(\theta)  + o(\varepsilon)\\
		= &   \varepsilon \int_{-\frac{\pi}{4}}^{\theta} \tzt'(\vartheta) [ m(\vartheta) \sin \vartheta+  m'(\vartheta)\cos \vartheta] d\vartheta - \varepsilon \frac{\tzt(\theta)}{h'(\theta)}(b)'(\theta)  + o(\varepsilon).
	\end{flalign*}
	
	For every $\theta \in [-\frac{\pi}{4}, \bar\theta]$, denote 
	$\Box_{0}(\theta) := m(\theta) + m''(\theta) - 3 h'(\theta) \cos\theta $, 
	$\Box_{1}(\theta) := \int_{-\frac{\pi}{4}}^{\theta} \tzt'(\vartheta) [\sin \vartheta m(\vartheta) + \cos \vartheta m'(\vartheta)] d\vartheta$, 
	and $\Box_{2}(\theta) := \cos \theta m(\theta) - \sin\theta m'(\theta) -a$.
	Therefore,
	\begin{flalign*}
	\frac{1}{2}\eqref{term:D2} = & \int_{\Omega_{1}^{-}}  \left(\vbu-\bu\right)(\Delta \bu -3) dx \\
		= & \int_{-\frac{\pi}{4}}^{\bar \theta} \int_{0}^{R(\theta)}  
		(\vbu(\bx(r, \theta)) - \bu(\bx(r, \theta))) 
		\left(\frac{m(\theta) + m''(\theta)}{h'(\theta) \cos\theta + r} -3 \right)
		(h'(\theta) \cos\theta + r) dr d\theta\\
		= & \int_{-\frac{\pi}{4}}^{\bar \theta} \int_{0}^{R(\theta)} 
		\bigg(\varepsilon \Box_{1}(\theta) - \varepsilon \frac{\tzt(\theta)}{h'(\theta)}(b)'(\theta)  + o(\varepsilon) \bigg)  
		\left( \Box_{0}(\theta) -3r \right) dr d\theta\\
		= & \varepsilon  \int_{-\frac{\pi}{4}}^{\bar \theta} 
		\bigg[ \Box_{1}(\theta) -  \frac{\tzt(\theta)(b)'(\theta)}{h'(\theta)} \bigg] 
		\left( \Box_{0}(\theta)R(\theta) -\frac{3}{2}R^2(\theta) \right) 
		d\theta  
		+ o(\varepsilon),
	\end{flalign*}
	and
	\begin{flalign*}
		\frac{1}{2}\eqref{term:D3} =	& \int_{
			\p \Omega_1^-\cap \p X}  \left(\vbu-\bu\right) \langle x-D\bu(x), \vec{n}(x)\rangle dS(x) \\
		= & \int_{-\frac{\pi}{4}}^{\bar \theta} \left(\vbu(\bx(0, \theta))-\bu(\bx(0, \theta))\right) \left( \frac{\partial \bu}{\partial x_1}(\bx(0, \theta)) -a \right) h'(\theta) d \theta \\
		= & \int_{-\frac{\pi}{4}}^{\bar \theta} 
		\bigg(\varepsilon \Box_{1}(\theta) - \varepsilon \frac{\tzt(\theta)}{h'(\theta)}(b)'(\theta)  + o(\varepsilon) \bigg)  
		\left( m(\theta) \cos \theta- m'(\theta) \sin \theta -a \right) h'(\theta) d \theta \\
		= & \varepsilon \int_{-\frac{\pi}{4}}^{\bar \theta} 
		\left( \Box_{1}(\theta) h'(\theta) -  \tzt(\theta)(b)'(\theta)  \right) 
		\Box_{2}(\theta)  d \theta
		+ o(\varepsilon).
	\end{flalign*}
	
	Note that $ \int_{\Omega_{1}^{\pm}\cup \Omega_{2}}  |D\left(\vbu(x)-\bu(x)\right)|^2 dx = o(\varepsilon)$. Since $\bu$ is optimal, $0 = \lim\limits_{\varepsilon \rightarrow 0 } \frac{\Phi[\vbu]- \Phi[\bu]}{\varepsilon}$ implies that 
	\begin{flalign*}
		0 =  & \int_{-\frac{\pi}{4}}^{\bar \theta} 
		\bigg[ \Box_{1}(\theta) -  \frac{\tzt(\theta)(b)'(\theta)}{h'(\theta)} \bigg] 
		\left( \Box_{0}(\theta)R(\theta) -\frac{3}{2}R^2(\theta) \right) 
		d\theta  	\\
		& + \int_{-\frac{\pi}{4}}^{\bar \theta} \left( \Box_{1}(\theta) h'(\theta) -  \tzt(\theta)(b)'(\theta)  \right) \Box_{2}(\theta)  d \theta\\
		= & \int_{-\frac{\pi}{4}}^{\bar \theta} \bigg\{  \Box_{1}(\theta) \left( \Box_{0}(\theta) R(\theta) - \frac{3R^2(\theta)}{2} + \Box_{2}(\theta)h'(\theta) \right) \\
		& \hspace{1cm} - \frac{\tzt(\theta)b'(\theta)}{h'(\theta)}  \left(  \Box_{0}(\theta)R(\theta) - \frac{3R^2(\theta)}{2} + \Box_{2}(\theta) h'(\theta) \right) 
		\bigg\} d \theta
	\end{flalign*}
	holds for any $\tzt: [-\frac{\pi}{4}, \bar\theta] \rightarrow \R$ such that $\tzt(-\frac{\pi}{4}) =\tzt(\bar \theta) =0$.
	
	Recall that $\Box_{1}(\theta) = \int_{-\frac{\pi}{4}}^{\theta}  \tzt'(\vartheta) [\sin \vartheta m(\vartheta) + \cos \vartheta m'(\vartheta)] d\vartheta {
		= \int_{-\frac{\pi}{4}}^{\theta}  \tzt'(\vartheta) \frac{b'(\vartheta)}{h'(\vartheta)} d\vartheta }$ and $\alpha(\theta) =   \Box_{0}(\theta)R(\theta) - \frac{3R^2(\theta)}{2} + \Box_{2}(\theta) h'(\theta)$. 
	Therefore, Fubini's theorem tells
	\begin{flalign*}
		0 
		=  -& \int_{-\frac{\pi}{4}}^{\bar \theta}  
		\left[\int_{-\frac{\pi}{4}}^{\bar \theta} \frac{\tzt(\theta)b'(\theta)}{h'(\theta)} 			
		-\int_{-\frac{\pi}{4}}^{\theta}  \tzt'(\vartheta) \frac{b'(\vartheta)}{h'(\vartheta)} d\vartheta\right] \alpha(\theta) d \theta
		\\= \phantom{-} & \int_{-\frac{\pi}{4}}^{\bar \theta}   \tzt'(\theta) \frac{b'(\theta)}{h'(\theta)} \int_{\theta}^{\bar \theta} \alpha(\vartheta) d\vartheta  d \theta
		- \int_{-\frac{\pi}{4}}^{\bar \theta} \frac{\tzt(\theta)b'(\theta)}{h'(\theta)}  \alpha(\theta)
		d \theta
		\\
		= - & \int_{-\frac{\pi}{4}}^{\bar \theta}   \tzt(\theta) 
		\bigg\{\frac{h'(\theta)b''(\theta) - b'(\theta)h''(\theta)}{(h'(\theta))^2} \int_{\theta}^{\bar \theta} \alpha(\vartheta) d\vartheta 
		- \frac{b'(\theta)}{h'(\theta)} \alpha(\theta) 
		\bigg\} d \theta 
		- \int_{-\frac{\pi}{4}}^{\bar \theta} \frac{\tzt(\theta)b'(\theta)}{h'(\theta)}  \alpha(\theta)
		d \theta
		\\
		= - &  \int_{-\frac{\pi}{4}}^{\bar \theta}   \tzt(\theta) 
		\bigg\{\cos\theta [ m(\theta) + m''(\theta) ] \int_{\theta}^{\bar \theta} \alpha(\vartheta) d\vartheta 
		\bigg\} d \theta 
	\end{flalign*}
	holds for any $\tzt \in C^2((-\frac{\pi}{4}, \bar\theta))$ of compact support, due to the locally uniform bound (b) $h'(\theta) >0$. Then the fundamental lemma of the calculus of variations tells
	\begin{flalign}
		\cos\theta [ m(\theta) + m''(\theta) ] \int_{\theta}^{\bar \theta} \alpha(\vartheta) d\vartheta = 0, \quad \forall \theta \in \left(-\frac{\pi}{4}, \bar\theta \right).
	\end{flalign}
	Recalling \eqref{Hessian}, the locally uniform rank $1$ property (c) of $D^2 \bu$ on $\Omega_{1}^{-}$ implies that $m(\theta) + m''(\theta) \ne 0$ for all $\theta \in \left(-\frac{\pi}{4}, \bar\theta \right)$.
	Since $\cos \theta >0 $ for all $\theta \in \left(-\frac{\pi}{4}, \bar\theta \right)$, one has
	\begin{flalign*}
		\int_{\theta}^{\bar \theta} \alpha(\vartheta) d\vartheta = 0, \quad \forall \theta \in \left(-\frac{\pi}{4}, \bar\theta \right).
	\end{flalign*}
	Thus,
	\begin{flalign*}
		\alpha(\theta) = 0, \quad \forall \theta \in \left(-\frac{\pi}{4}, \bar\theta \right).	
	\end{flalign*}
\end{proof}

\medskip

\section{On nonsmooth convex ruled surfaces. By Cale Rankin}
\label{section:Rankin}

In this appendix we confirm that the graph of the optimizing indirect utility is a ruled surface in the region where $\det (D^2u) = 0$ holds. For smooth convex $u$ this is a classical fact, known to geometers such as Monge since the eighteenth century: if $u \in C^2(X)$ one obtains the rulings (segments) along which $u$ is affine
by integrating the continuous vector field given by zero-eigenvalue eigenvectors of the Hessian $D^2 u$.   Our present goal is to extend this result to the merely $C^{1,1}_{loc}$ regularity \eqref{regularity} guaranteed for the optimal indirect utility $\bar u$ by the work of \cite{CaffarelliLions06+}, also described in \cite{mccann2023c}.  For this we recall the definition of the Monge-Amp\`ere measure for $C^{1}$ convex functions. The Monge-Amp\`ere measure, denoted $\mu_u$, of such a function is the measure defined for $E \subset \R^n$ by
\[ \mu_u(E) := \mathop{\rm vol}[\p u(E)] = \mathop{\rm vol}[Du(E)].\]
For a proof that this quantity is a measure and other basic properties see the book of \cite{Figalli17}. The Monge-Amp\`ere measure extends the measure $\det D^2u \ dx$ to functions which are not twice differentiable and we employ it in conjunction with the Aleksandrov maximum principle, stated here as in Theorem 2.8 of \cite{Figalli17}.

\begin{theorem}[The Monge-Amp\`ere measure controls deviation from linearity]
\label{T:Aleksandrov maximum principle}
  Let $X \subset \R^n$ be an arbitrary convex set and $u:X\rightarrow \mathbf{R}$ a convex function. Assume the restriction of $u$ to $\partial X$ is affine, that is there is $l(x) = p \cdot x + a $ such that $u \equiv l$ on $\partial X$. Then there is a constant $C$ depending only on the dimension such that for all $x \in X$ the following estimate holds 
\[|u(x) - l(x)|^n \leq C \text{diam}(X)^{n-1}\text{dist}(x,\partial X) \mu_u(X). \]
\end{theorem}

This theorem simplifies considerably when $u$ is $C^{1,1}_{\text{loc}}$. Indeed, the Lipschitz continuity of the first derivatives implies almost everywhere second differentiability. Subsequently the Monge-Amp\`ere measure is absolutely continuous with respect to Lebesgue and given by $\mu_u = \det D^2u(x) \ dx $ where we may set $D^2u(x) = 0$ at points where $u$ is not twice differentiable \cite[Lemma 2.3]{TrudingerWang08}. Now we extend the classical result on ruled surfaces as follows.

\begin{lemma}[Ruled surface]\label{L:ruled surfaces}
  Let $u \in C^{1,1}_{\text{loc}}(\Omega)$ be the restriction of a convex function $u:\mathbf{R}^2 \rightarrow \mathbf{R} \cup \{+\infty\}$ which, throughout a bounded open set $\Omega \subset \mathbf{R}^2$, satisfies $\det D^2u = 0 < \Delta u$ at points of second differentiability. If $y \in Du(\Omega)$, each connected component of $Du^{-1}(y) \cap \Omega$ is a line segment with endpoints on $\partial\Omega$.
\end{lemma}
\begin{proof}
Take $y_0 = Du(x_0)$ for some $x_0 \in \Omega$. Let $l$ denote the corresponding support $l(x):= u(x_0) + y_0\cdot(x-x_0)$. We claim the intersection of the convex set $S_0 := \{u \equiv l\} = Du^{-1}(y_0)$ with $\Omega$ consists of line segments with endpoints on $\partial\Omega$.
  
  First note $x_0 \in S_0 \cap \Omega$.  Now $\Delta u > 0$ implies $u$ is not affine on any open set and thus $S_0$ is at most one dimensional. Next, 
  $\det D^2u = 0$ a.e. on $\Omega$ combines with $u \in C^{1,1}_{loc}(\Omega)$ to imply the Monge-Amp\`ere measure of $u$ vanishes on $\Omega$, that is $\mu_u (\Omega)= 0$.  For a contradiction suppose a connected component of $S_0 \cap \Omega$  does not intersect the boundary. To avoid a contradiction, the convex set
  \[ S_\epsilon := \{ x\in\R^n : u(x) \le l(x) +\epsilon \}\]
which contains $x_0 \in S_0\cap \Omega$ 
   must be strictly contained in $\Omega$ for $\epsilon > 0$ sufficiently small. The Aleksandrov maximum principle 
  implies the contradiction
  via  Theorem \ref{T:Aleksandrov maximum principle}.
  
Thus $Du^{-1}(y)$ is a line segment with an endpoint on $\partial\Omega$. To see both endpoints lie on $\partial\Omega$, suppose otherwise. Without loss of generality $x_0 = 0$ and $Du^{-1}(y) = \{te_1; a \leq t \leq b\}$ with $b>0$ and $be_1 \in \text{int}(\Omega)$ . Then, by tilting the support, we see for $\epsilon>0$ sufficiently small the convex set 
  \[ S'_\epsilon = \{x \in \R^n : u(x) \leq l(x) + \epsilon x_1\},\]
  again contains $x_0$ and
is compactly contained in $\Omega$. This yields the same contradiction to the Aleksandrov maximum principle as above. Thus $Du^{-1}(y)$ is a line segment with both endpoints on $\partial\Omega$. 
\end{proof}

\bibliographystyle{apacite}
\bibliography{newbib230119.bib}

\end{document}